\documentclass[english,12pt,oneside]{amsproc}
\usepackage[english]{babel}
\usepackage{a4wide}
\usepackage{amsthm}
\usepackage{graphics}
\usepackage{amsfonts, amssymb, amscd, amsmath}
\usepackage{mathdots}
\usepackage{latexsym}
\usepackage[matrix,arrow,curve]{xy}
\usepackage{mathabx,mathtools}
\usepackage{color}
\usepackage{pbox}
\usepackage{tikz}
\usetikzlibrary{matrix,decorations.pathreplacing,positioning}
\usepackage{hyperref}

\DeclareMathOperator{\id}{id}

\DeclareMathOperator{\Hom}{Hom}
\DeclareMathOperator{\rk}{rk}
\DeclareMathOperator{\odd}{odd}

\DeclareMathOperator{\ord}{ord}
\DeclareMathOperator{\Flats}{Flats}

\DeclareMathOperator{\str}{star}

\DeclareMathOperator{\Cy}{Cy}
\DeclareMathOperator{\St}{St}

\DeclareMathOperator{\diag}{diag}
\DeclareMathOperator{\Fl}{Fl}

\DeclareMathOperator{\Pe}{Pe}

\DeclareMathOperator{\Cayley}{Cayley}

\DeclareMathOperator{\Cl}{Cl}
\DeclareMathOperator{\Gr}{Gr}
\DeclareMathOperator{\pr}{pr}

\DeclareMathOperator{\clos}{clos}
\DeclareMathOperator{\compl}{compl}

\DeclareMathOperator{\conn}{conn}
\DeclareMathOperator{\sbgrps}{Sbgrps}
\DeclareMathOperator{\Facets}{Facets}
\DeclareMathOperator{\pa}{part}

\newcommand{\bd}{\bar{k}}

\newcommand{\simc}{\!\!\sim}

\newcommand{\PT}{\mathcal{PT}}

\newcommand{\CP}{\mathbb{C}P}

\newcommand{\ca}[1]{\mathcal{#1}}
\newcommand{\Hr}{\widetilde{H}}

\newcommand{\Ca}{\mathcal{C}}
\newcommand{\F}{\mathcal{F}}
\newcommand{\Zt}{\mathbb{Z}_2}
\newcommand{\I}{\mathbb{I}}

\newcommand{\Mbd}{{}^{\bd}\!M}

\newcommand{\uuu}{\mathfrak{u}}
\newcommand{\g}{\mathfrak{g}}
\newcommand{\ttt}{\mathfrak{t}}

\def\Co{\mathbb C}
\def\Ro{\mathbb R}
\def\Qo{\mathbb Q}
\def\Zo{\mathbb Z}

\def\ko{\Bbbk}

\newcounter{stmcounter}[section]
\newcounter{thcounter}
\newcounter{problcounter}

\numberwithin{equation}{section}

\theoremstyle{plain}
\newtheorem{cor}[stmcounter]{Corollary}

\newtheorem{thm}[thcounter]{Theorem}

\newtheorem{prop}[stmcounter]{Proposition}
\newtheorem{lem}[stmcounter]{Lemma}
\newtheorem{probl}[problcounter]{Problem}

\newtheorem{clai}[stmcounter]{Claim}

\theoremstyle{definition}
\newtheorem{defin}[stmcounter]{Definition}

\theoremstyle{remark}
\newtheorem{ex}[stmcounter]{Example}
\newtheorem{rem}[stmcounter]{Remark}
\newtheorem{con}[stmcounter]{Construction}

\begin{document}

\title[Torus actions on submanifolds of flag varieties]{Cluster-permutohedra and submanifolds of flag varieties with torus actions}

\author{Anton Ayzenberg}
\address{Faculty of computer science, National Research University Higher School of Economics, Moscow, Russian Federation}
\email{ayzenberga@gmail.com}

\author{Victor Buchstaber}
\address{V.A. Steklov Mathematical Institute, RAS, and Faculty of computer science, National Research University Higher School of Economics, Moscow, Russian Federation}
\email{buchstab@mi-ras.ru}

\thanks{The article was prepared within the framework of the HSE University Basic Research Program}

\keywords{isospectral Hermitian matrices, sparse matrices, graphicahedron, cluster-permutohedron, compact torus action, faces of torus actions, full flag variety, partial flag variety, Grassmann manifold}

\subjclass[2020]{Primary: 57S12, 14M15, 05E10, 05E18, 06A15 Secondary: 05B35, 55R20, 55T10, 20B35, 06A07, 57S15, 05C15}

\begin{abstract}
In this paper we describe a relation between the notion of graphicahedron, introduced by Araujo-Pardo, Del R\'{\i}o-Francos, L\'{o}pez-Dudet, Oliveros, and Schulte in 2010, and toric topology of manifolds of sparse isospectral Hermitian matrices. More precisely, we recall the notion of a cluster-permutohedron, a certain finite poset defined for a simple graph $\Gamma$. This poset is build as a combination of cosets of the symmetric group, and the geometric lattice of the graphical matroid of $\Gamma$. This poset is similar to the graphicahedron of $\Gamma$, in particular, 1-skeleta of both posets are isomorphic to Cayley graphs of the symmetric group. We describe the relation between cluster-permutohedron and graphicahedron using Galois connection and the notion of a core of a finite topology. We further prove that the face poset of the natural torus action on the manifold of isospectral $\Gamma$-shaped Hermitian matrices is isomorphic to the cluster-permutohedron. Using recent results in toric topology, we show that homotopy properties of graphicahedra may serve an obstruction to equivariant formality of isospectral matrix manifolds. We introduce a generalization of a cluster-permutohedron and describe the combinatorial structure of a large family of manifolds with torus actions, including Grassmann manifolds and partial flag manifolds.
\end{abstract}

\maketitle

\section{Introduction}\label{secIntro}

Let $\Gamma$ be a simple graph with a vertex set $[n]=\{1,2,\ldots,n\}$ and an edge-set $E_\Gamma$. It is assumed that $\Gamma$ is connected. Four objects associated with $\Gamma$ are in the focus of our paper: three of them are combinatorial and the last one is topological, necessary definitions are given in the latter sections.
\begin{enumerate}
\item The Cayley graph $\Cayley_\Gamma$ of the symmetric group $\Sigma_n$ with the set of transpositions $\{(i, j)\mid \{i,j\}\in E_\Gamma\}$ chosen as the generating set.
\item The \emph{graphicahedron} $\Gr_\Gamma$.
\item The \emph{cluster-permutohedron} $\Cl_\Gamma$.
\item The manifold $M_{\Gamma,\lambda}$ of $\Gamma$-shaped Hermitian matrices having a generic fixed spectrum $\lambda$. This manifold carries a natural effective action of a compact $(n-1)$-dimensional torus $T$. We call manifolds $M_{\Gamma,\lambda}$ isospectral matrix manifolds.
\end{enumerate}

The graphicahedron is a graded poset (see Definition~\ref{defGraphicahedron}), it was introduced in~\cite{Araujo-Pardo_Rio-Francos_Lopez-Dudet_Oliveros_Schulte:2010} as a higher dimensional structure extending the Cayley graph $\Cayley_\Gamma$. Cluster-permutohedra were introduced by the authors in~\cite{Ayzenberg_Buchstaber_MS:2021} in relation with the study of toric topology of the manifolds $M_{\Gamma,\lambda}$. In that paper, we assumed $\Gamma$ is a tree, in which case the cluster-permutohedron $\Cl_\Gamma$ is a dually simplicial poset. In the case of a tree, the torus action on $M_{\Gamma,\lambda}$ has complexity zero, so its orbit space $Q_{\Gamma,\lambda}=M_{\Gamma,\lambda}/T$ is a manifold with corners. We proved that the face poset of $Q_{\Gamma,\lambda}$ is isomorphic to the cluster-permutohedron $\Cl_\Gamma$. At the time we published the paper~\cite{Ayzenberg_Buchstaber_MS:2021}, we were unaware of the papers~\cite{Araujo-Pardo_Rio-Francos_Lopez-Dudet_Oliveros_Schulte:2010,Rio-Francos_Hubard_Oliveros_Schulte:2012} studying graphicahedra which were certainly related to our research.

The current paper aims to bring together the theory of graphicahedra and the theory of cluster-permutohedra for general graphs. We apply some known results on graphicahedra to study equivariant formality of isospectral matrix manifolds. Our previous work~\cite{Ayzenberg_Buchstaber_IMRN:2021} on the relation between isospectral matrix manifolds and Hessenberg varieties recently gained attention of the specialists in representation theory and symmetric functions~\cite{Abe_Horiguchi:2019,Masuda_Sato:2022,Precup_Sommers:2022} in context of the possible relations to e-positivity and Stanley--Stembridge conjecture. We hope that the current paper may also reveal new connections between algebraic combinatorics and toric topology. Our result describes the relations between objects listed above.

\begin{thm}\label{thmMainAll}
The following hold true.
\begin{enumerate}
  \item For any graph $\Gamma$ and any generic spectrum $\lambda$, the face poset of the natural torus action on $M_{\Gamma,\lambda}$ is isomorphic to the cluster-permutohedron $\Cl_\Gamma$.
  \item $\Cl_\Gamma$ is a locally geometric poset (see Definition~\ref{defLocGeomPoset}). Its 1-skeleton is isomorphic to the Cayley graph $\Cayley_\Gamma$. For any element $x\in\Cl_\Gamma$ of rank $0$, the upper order ideal $(\Cl_\Gamma)_{\geqslant x}$ is isomorphic to the geometric lattice of the graphical matroid of $\Gamma$.
  \item For any graph $\Gamma$, there exists a Galois connection $\Cl_\Gamma\rightleftarrows \Gr_\Gamma$ preserving skeleta (the skeleta of $\Gr_\Gamma$ should be defined appropriately, see Construction~\ref{conSkeleta}). The skeleton $(\Cl_\Gamma)_j$ is the core of the skeleton $(\Gr_\Gamma)_j$ for $j<\rk \Cl_\Gamma$.
  \item Let $g$ be the girth of $\Gamma$, that is the minimal cycle length. Then the skeleta $(\Cl_\Gamma)_{g-2}$ and $(\Gr_\Gamma)_{g-2}$ are isomorphic. In particular, when $\Gamma$ is a tree (in which case $g$ is assumed $+\infty$) there exists a canonical isomorphism $\Cl_\Gamma\cong\Gr_\Gamma$.
  \item The torus action on $M_{\Gamma,\lambda}$ is $(g-1)$-independent (see Definition~\ref{defJindependent}).
\end{enumerate}
\end{thm}

Notice that both the theory of graphicahedra and the theory of cluster-permutohedra originated from the particular important example. If $\Gamma=\I_n$ is the path graph, both $\Gr_{\I_n}$ and $\Cl_{\I_n}$ are isomorphic to the poset of faces of the permutohedron $\Pe^{n-1}$. The 1-skeleton of $\Pe^{n-1}$ is isomorphic to the Cayley graph of $\Sigma_n$ with the standard set of generators $(1,2),(2,3),\ldots,(n-1,n)$, and $\Gr_{\I_n}$ is considered as the higher-dimensional structure over this Cayley graph. On the other hand, $\Cl_{\I_n}$ is isomorphic to the face poset of the manifold of isospectral tridiagonal Hermitian matrices. This fact was proved in the classical work of Tomei~\cite{Tomei:1984} who considered the real version of this manifold. We call the manifold $M_{\I_n,\lambda}$ the Hermitian Tomei manifold. The appearance of a convex permutohedron has gained a symplectic explanation in the work of Bloch--Flaschka--Ratiu~\cite{Bloch_Flaschka_Ratiu:1990}. All items of Theorem~\ref{thmMainAll} are well known for the path graph $\I_n$.

After introducing all notions and proving Theorem~\ref{thmMainAll}, we switch to the applications of some known results about graphicahedra in toric topology in Section~\ref{secNonFormality}. Our goal is not to prove new results there, but rather establish new relations between different areas of mathematics.

We concentrate on two types of graphs: the cycle graphs $\Cy_n$, $n\geqslant 3$, and the complete bipartite graph $K_{1,3}$ (which is also called the star graph $\St_3$ or the claw-graph in various sources), and study the topology of the corresponding graphicahedra and cluster-permutohedra. Notice that $|\Gr_\Gamma|$ is contractible. Indeed, $\Gr_\Gamma$ always has the top element $\hat{1}$, so its geometrical realization is a cone. We should drop the top element if we want the topological study of our poset to make sense. The following statement was proved in~\cite{Rio-Francos_Hubard_Oliveros_Schulte:2012}.

\begin{prop}[\cite{Rio-Francos_Hubard_Oliveros_Schulte:2012}]\label{propSymGraphicahedra}\hfill
\begin{enumerate}
  \item $\Gr_{\St_3}\setminus\{\hat{1}\}$ is the face poset of a cell subdivision of $T^2$. This subdivision is the quotient of the hexagonal tiling of the plane $\Ro^2$ by a certain pair of chiral vectors~\footnote{The naming ``chiral vectors'' for the vectors of periods is borrowed from chemistry, where it is used in the combinatorial study of graphen and nanotubes, see~\cite{Buchstaber_Erokhovets:2016}}.
  \item $\Gr_{\Cy_n}\setminus\{\hat{1}\}$, for $n\geqslant 3$, is a cell subdivision of an $(n-1)$-dimensional torus $T^{n-1}$. This subdivision arises as the quotient of the regular permutohedral tiling of the space $\Ro^{n-1}$ by a certain lattice.
\end{enumerate}
\end{prop}

In our paper we give a topological application of this proposition. More precisely, we study equivariant formality property of isospectral matrix manifolds. The property of equivariant formality is the basis of GKM theory. Equivariant formality holds automatically for algebraic torus actions on smooth projective varieties (this is the setting studied originally by  Goresky, Kottwitz, and MacPherson). However, some natural manifolds with a compact torus action may lack equivariant formality, in which case GKM-theory cannot be applied to describe their cohomology. So far, it is important to understand, whether a manifold is equivariantly formal or not. At the first place, we are interested in the manifolds $M_{\Gamma,\lambda}$, and generally they are not complex projective varieties.

For both types of graphs described in Proposition~\ref{propSymGraphicahedra}, the toroidal nature of the poset implies $H^1(\Gr_{\Gamma}\setminus\{\hat{1}\};\Zo)\neq 0$, and therefore
\begin{equation}\label{eqNonAcyclic}
H^1((\Cl_{\Gamma})_2;\Zo)\neq 0
\end{equation}
for $\Gamma=\St_3$ or $\Cy_n$, $n\geqslant 4$, according to general connections between $\Gr$ and $\Cl$ given by Theorem~\ref{thmMainAll}. In a recent paper~\cite{Ayzenberg_Masuda_Solomadin:2022} we proved that, whenever a torus action on a manifold $X$ is cohomologically equivariantly formal, the skeleta of its face poset $S(X)$ satisfy certain acyclicity condition, see Proposition~\ref{propSkeletonAcyclic}. Formula~\eqref{eqNonAcyclic} violates this condition for $\Gamma=\St_3$ or $\Cy_n$, $n\geqslant 4$. Therefore we conclude that the isospectral matrix manifolds $M_{\St_3,\lambda}$ and $M_{\Cy_n,\lambda}$, $n\geqslant 4$, are not equivariantly formal.

This result is already known actually. The 6-dimensional manifold $M_{\St_3,\lambda}$ was studied in detail by the authors in~\cite{Ayzenberg_Buchstaber_MS:2021}, while $M_{\Cy_n,\lambda}$ was studied by the first author in~\cite{Ayzenberg:2020}. To prove non-formality, we also studied the face structures of these manifolds, in particular, we independently proved Proposition~\ref{propSymGraphicahedra} in our papers. It should be noticed that the simplicial poset dual to $\Gr_{\Cy_n}\setminus\{\hat{1}\}$ is vertex-minimal among all simplicial cell subdivisions of the torus $T^{n-1}$, which relates this object to the theory of crystallizations~\cite{Ferri_Gagliardi_Grasselli:1986}. The particular poset $\Gr_{\Cy_3}\setminus\{\hat{1}\}$ --- the subdivision of $T^2$ into 3 hexagons, representing the embedding of $K_{3,3}$ into $T^2$, --- and its relation to isospectral matrix spaces, first appeared in the work of van Moerbeke~\cite{vanMoerbeke:1976} on integrable dynamical systems. A construction similar to the graphicahedron $\Gr_{\Cy_n}$ was introduced by Panina~\cite{Panina:2015} by the name \emph{cyclopermutohedron}. In particular, the second item of Proposition~\ref{propSymGraphicahedra} was also proved in~\cite{Nekrasov_Panina_Zhukova:2016} with the use of discrete Morse theory.

Finally, we proceed by generalizing cluster-permutohedra. The generalization allows to describe face posets of a big family of torus-stable submanifolds in partial flag varieties. This family includes all isospectral matrix manifolds, partial flag varieties (including Grassmann manifolds), regular semisimple Hessenberg varieties, their generalizations (which may fail to be complex varieties in general), and their analogues in the partial flag varieties. The main contribution here is the universal construction of \emph{symmetric identification sets}, which underlines similarities between cluster-permutohedra, graphicahedra, and a bunch of objects studied in toric topology, such as small covers, quasitoric manifolds, and moment-angle manifolds. Our generalization of cluster-permutohedron may be considered as a higher-dimensional version of the Schreier coset graphs~\cite{Schreier:1927} of the symmetric group, --- in a similar way as the usual cluster-permutohedra appear to be the higher-dimensional versions of Cayley graphs.

The paper has the following organization. Section~\ref{secDefs} contains all the required definitions, separated into three subjects: discrete mathematics, toric topology, and isospectral matrix spaces. We also describe the general construction of twin submanifolds in a full flag variety which first appeared in our work~\cite{Ayzenberg_Buchstaber_IMRN:2021}. The twin of $M_{\Gamma,\lambda}$ is a certain generalization of a regular semisimple Hessenberg variety. In Section~\ref{secProofs} we prove all points of Theorem~\ref{thmMainAll}. In Section~\ref{secNonFormality} we deduce non-formality of several types of matrix manifolds using some known results of toric topology and structural results about graphicahedra and cluster-permutohedra. The setting is generalized to arbitrary semisimple Lie types, and we pose several general problems in this setting. In Section~\ref{secIdentifications} we discuss the results and provide the unifying construction of symmetric identification sets. This construction is applied, in particular, to describe face posets of a large family of manifolds, including, in particular, Grassmann manifolds, partial flags, regular semisimple Hessenberg manifolds, and their partial flags' analogues.

\section{Definitions}\label{secDefs}

\subsection{Posets}\label{subsecCombinatorics}

Let us recall several definitions relevant to our work. We deal with partially ordered sets (posets for short) which are always assumed finite. For a poset $S$ and an element $s\in S$ we denote by $S_{\geqslant s}$ the upper order ideal $\{t\in S\mid t\geqslant s\}$, and similarly for $S_{\leqslant s}$, $S_{>s}$, $S_{<s}$. A minimal element is an element $x\in S$ such that there is no $s<x$ in $S$. The bottom (otherwise called the least) element $\hat{0}$ is the element with the property $\hat{0}\leqslant s$ for any $s\in S$. Similarly, one defines maximal elements, and the top (the greatest) element~$\hat{1}$. Minimal and maximal elements always exist. If the top element exists, it is unique, and similarly for the bottom element. Minimal elements will usually be denoted by $x$.

For any poset $S$ one can construct the simplicial complex $\ord(S)$, the order complex. Vertices of $\ord(S)$ are all the elements of $S$, and simplices are chains (linearly ordered subposets) of $S$. The geometrical realization of $\ord(S)$ is called the geometrical realization of $S$ and denoted by $|S|$. This is a standard way of turning poset into a compact Hausdorff topological space. If $S$ has a top or a bottom element, its order complex is a cone, hence $|S|$ is contractible.

An element $s$ is said to cover an element $t$ if $t<s$ and there is no element $r\in S$ such that $t<r<s$. A poset $S$ is called graded if there exists a function $\rk\colon S\to \Zo_{\geqslant 0}$ such that $\rk(x)=0$ for any minimal element $x\in S$, and $\rk s=\rk t+1$ whenever $s$ covers $t$. The rank of a graded poset is the maximal rank of its elements. If a graded poset has the bottom element $\hat{0}$, then all elements covering $\hat{0}$ (that are all elements of rank $1$) are called atoms.

A poset is called a lattice if there exist $\hat{0}$, $\hat{1}$, as well as the least upper bounds (joins) and the greatest lower bounds (meets) of any sets of elements. A graded poset $S$ is called a \emph{geometric lattice} if it satisfies
\begin{enumerate}
  \item Every element $s\in S$ is a join of some set of atoms (which essentially means that $S$ can be considered a subposet of the Boolean cube generated by the set of atoms).
  \item The rank function satisfies $\rk s+\rk t\geqslant \rk(s\vee t)+\rk (s\wedge t)$.
\end{enumerate}

Geometric lattices appear as the posets of flats of matroids. We will not need the definition of general matroids, but give the construction of geometric lattices of flats of linear and graphical matroids. General theory can be found in~\cite{Oxley:1992}.

\begin{con}\label{conLinearMatroids}
A collection $\beta=\{\beta_1,\ldots,\beta_m\}$ of vectors in a vector space $U$ (over any field), is called \emph{a linear matroid}. A \emph{flat} of $\beta$ is an intersection of $\beta$ with any vector subspace $\Pi\subseteq U$. The flats are naturally ordered by inclusion, and the poset of flats is denoted $\Flats(\beta)$. Each flat comes equipped with its rank: the rank of a flat $A$ is given by the dimension of its linear span. This makes $\Flats(\beta)$ a graded poset. Actually, $\Flats(\beta)$ is a geometric lattice.
\end{con}

\begin{con}\label{conGeomLatticeOfGraph}
If $\Gamma=([n],E_\Gamma)$ is a simple graph, one can define a linear matroid as follows. For any field $\ko$ consider the arithmetic vector space $\ko^n$ with the basis $e_1,\ldots,e_n$. Consider the matroid
\[
\beta(\Gamma)=\{e_i-e_j\mid \{i,j\}\in E_\Gamma\}.
\]
The lattice of flats $\Flats(\beta(\Gamma))$ is denoted simply by $\Flats(\Gamma)$ and called \emph{the geometric lattice of the graph} $\Gamma$.

The greatest flat $\hat{1}$ of $\Flats(\Gamma)$ is the linear span of $\beta(\Gamma)$, which coincides with the hyperplane $\left\{\sum x_i=0\right\}\subset\ko^n$ (for connected graphs). Hence $\rk\Flats(\Gamma)=n-1$.
\end{con}

We will also need the following definition introduced in~\cite{Ayzenberg_Cherepanov:2022} as a far-reaching generalization of (dually) simplicial posets.

\begin{defin}\label{defLocGeomPoset}
A poset $S$ is called \emph{locally geometric} if it has the top element $\hat{1}$, and for any element $s\in S$, the upper order ideal $S_{\geqslant s}$ is a geometric lattice.
\end{defin}

Let us recall another set of definitions and constructions. Recall that a map $f\colon S\to T$ between two posets is called monotonic if $s_1\leqslant s_2$ in $S$ implies $f(s_1)\leqslant f(s_2)$ in $T$. All maps of posets are assumed monotonic in the following. If there are two maps $f\colon S\to T$ and $g\colon T\to S$, we simply write them as $f\colon S\rightleftarrows T\colon g$.

\begin{defin}\label{defGaloisCon}
A pair of maps $f\colon S\rightleftarrows T\colon g$ is called a \emph{Galois connection} (an order preserving Galois connection), if one of the two equivalent properties holds
\begin{enumerate}
  \item For any $s\in S$ and $t\in T$ the inequality $f(s)\leqslant t$ is equivalent to $s\leqslant g(t)$.
  \item For any $s\in S$ there holds $s\leqslant g(f(s))$, and for any $t\in T$ there holds $f(g(t))\leqslant t$.
\end{enumerate}
\end{defin}

If posets are considered as small categories, then the monotonic maps are functors, and Galois connection is a pair of adjoint functors. A particular case of Galois connection appears quite often both in theoretical studies and in information retrieval theory.

\begin{con}\label{conClosureOperator}
Let $T$ be a subposet of $S$, and $\iota\colon T\to S$ the canonical inclusion. Suppose that there exists $\pi\colon S\to T$ such that $\pi\colon S\rightleftarrows T\colon \iota$ is a Galois connection. In this case it is said that there is a Galois insertion of $T$ in $S$. The composite map $\clos=\iota\circ\pi\colon S\to S$ satisfies the properties
\begin{enumerate}
  \item $s\leqslant\clos(s)$;
  \item $s_1\leqslant s_2$ implies $\clos(s_1)\leqslant\clos(s_2)$;
  \item $\clos(\clos(s))=\clos(s)$.
\end{enumerate}
In other words, $\clos$ is a closure operator on $S$, the notion well-known in general topology and logic. So far, $T$ coincides with the set of closed elements of $S$ under $\clos$. On the contrary, given any closure operator $\clos\colon S\to S$, one can construct a Galois insertion of the subset of closed elements.
\end{con}

The variations of the next lemma are known. In particular, Barmak~\cite{Barmak:2011} proved a variant of this lemma for simple homotopy equivalence. We give an idea of proof for completeness of exposition.

\begin{lem}\label{lemGaloisHomotopy}
If $f\colon S\rightleftarrows T\colon g$ is a Galois connection, then the induced maps of geometrical realizations $f\colon |S|\rightleftarrows |T|\colon g$ deliver homotopy equivalence between $|S|$ and $|T|$.
\end{lem}

\begin{proof}
One needs to prove the homotopy equivalences $g\circ f\simeq \id_S$ and $f\circ g\simeq\id_T$. Let us prove the first one, the second being completely similar. The composite map $\clos=g\circ f$ is a closure operator on $S$, so we need to construct a continuous homotopy $F\colon |S|\times[0,1]\to|S|$ connecting $\id_S$ to $\clos$. For finite posets the homotopy can be done by consecutively moving each vertex $s\in S\setminus\clos(S)$ to its closure $\clos(s)$. If we start with larger $s$ and proceed from top to bottom, then all elements of $S\setminus\clos(S)$ will be eventually collapsed to $\clos(S)$. The details are left as an exercise.
\end{proof}

Another block of definitions is motivated by the study of finite topologies. The definitions are usually formulated for finite topologies, but since the category of finite posets is equivalent to the category of finite topological spaces satisfying Kolmogorov T0 separation axiom, we give the definitions for posets, without digging into finite topologies. As before, $S$ denotes a finite poset.

\begin{defin}\label{defCore}
An element $s\in S$ is called an \emph{upbeat} if it is covered by exactly one element. An element $s\in S$ is called a \emph{downbeat} if it covers exactly one element. If the poset $S$ has neither downbeats nor upbeats, it is called a \emph{core}. If $T$ is obtained from $S$ by consecutive removals of upbeats and downbeats, and $T$ is a core, then $T$ is called a core of $S$.
\end{defin}

In this definition we used the terminology of May~\cite{May:2019}. In the original paper of Stong~\cite{Stong:1966} upbeats are called \emph{linear}, and downbeats are called \emph{colinear} elements. It follows from the results of Stong on finite topologies, that all cores of $S$ are isomorphic, so we can speak about \emph{the core} of $S$. We provide another useful lemma which we consider a folklore.

\begin{lem}\label{lemClosureUpbeatRemovals}
Let $\clos\colon S\to S$ be a closure operator, and $T=\clos(S)$ a subset of closed elements. Then $T$ is obtained from $S$ by consecutive upbeats' removals.
\end{lem}

\begin{proof}
Consider the subposet $R=S\setminus\clos(S)$. Any maximal element $s$ of $R$ is an upbeat in $S$, since it is covered only by $\clos(s)$. Indeed, we have $\clos(s)>s$, and if there exists another closed element $s'$ with $s'>s$, then $s'=\clos(s')\geqslant\clos(s)$, hence $s'$ does not cover $s$. Remove $s$ from $R$. Repeat the argument until $R$ is exhausted.
\end{proof}

\begin{lem}\label{lemCoreGaloisInsertion}
If $\iota\colon T\to S$ is a Galois insertion, and $T$ is a core, then $T$ is the core of~$S$.
\end{lem}

This follows directly from the definition of the core and the previous lemma.

There is an example of Galois connection relevant to our study.

\begin{ex}\label{exGaloisInsertGeomLattice}
Let $\beta=\{\beta_1,\ldots,\beta_m\}$ be a linear matroid. Let $2^{\beta}$ denote the Boolean lattice of subsets of $\beta$, and $\Flats(\beta)$ be the lattice of flats, as before. Then we have a Galois insertion
\[
\iota\colon \Flats(\beta) \hookrightarrow 2^{\beta},\quad \pi\colon 2^{\beta}\to \Flats(\beta).
\]
The map $\iota$ outputs the set of all vectors lying in a given flat. The map $\pi$ outputs the span of $W$, that is the intersection of $\beta$ with the subspace $\langle W\rangle$.
\end{ex}

\begin{rem}\label{remBeatsInLattice}
Any geometric lattice has beats: for example, all atoms are downbeats, and all coatoms (elements covered by $\hat{1}$) are upbeats. However, the following statement holds.
\end{rem}

\begin{lem}\label{lemGeomLatticeIsCore}
If $S$ is a geometric lattice, then $S\setminus\{\hat{0},\hat{1}\}$ is a core.
\end{lem}

\begin{proof}
Since reversing the order in a geometric lattice produces another geometric lattice, it is sufficient to prove that there are no downbeats in $S\setminus\{\hat{0},\hat{1}\}$. Assuming the contrary, there is an element $s$ which covers the unique element $s'$. By the definition of the lattice, $s$ is a join (a supremum) of some nonempty set of atoms $x_1,\ldots,x_l$. Then we also have $s'\geqslant x_1,\ldots,x_l$, which contradicts to the definition of supremum.
\end{proof}

\subsection{Posets related to a graph}\label{subsecPosetsFromGraph}

Now we switch attention to the specific posets: cluster-permutohedra and graphicahedra. As before, let $\Gamma$ be a graph (assumed connected for simplicity) on the vertex set $V_\Gamma=[n]=\{1,2,\ldots,n\}$ with an edge set $E_\Gamma$.

\begin{con}\label{conClusterPerm}
An unordered subdivision $\ca{C}=\{V_1,\ldots,V_k\}$ of the set $V_\Gamma$ will be called a \emph{clustering} if each induced subgraph $\Gamma_{V_i}$ is connected. The set $\ca{L}_\Gamma$ of all clusterings is partially ordered by refinement: $\ca{C}'\leq \ca{C}$ if each $V_i'\in \ca{C}'$ is a subset of some $V_j\in \ca{C}$. The poset $\ca{L}_\Gamma$ has the least element $\hat{0}=\{\{1\},\ldots,\{n\}\}$ and the greatest element $\hat{1}=\{V_\Gamma\}$. The poset $\ca{L}_\Gamma$ is graded by $\rk\ca{C}=n-|\ca{C}|$. It can be seen that $\ca{L}_\Gamma$ is a geometric lattice, more precisely it is exactly the lattice of flats $\Flats(\Gamma)$ of the graphical matroid corresponding to $\Gamma$, see Construction~\ref{conGeomLatticeOfGraph}.

Let us consider all possible bijections $p$ from $V_\Gamma$ to $[n]=\{1,\ldots,n\}$. We say that two bijections $p_1,p_2\colon V\to[n]$ are equivalent with respect to a clustering $\ca{C}$ (or $\ca{C}$-\emph{equivalent}), denoted $p_1\stackrel{\ca{C}}{\sim}p_2$, if $p_1,p_2$ differ by permutations within clusters $V_i$ of $\ca{C}$. In other words,
\[
p_1^{-1}\circ p_2 \in \Sigma_{\ca{C}}=\Sigma_{V_1}\times\cdots\times\Sigma_{V_k}\subseteq\Sigma_V.
\]
A class of $\ca{C}$-equivalent bijections will be called \emph{an assignment} for the clustering $\ca{C}$. Thus far, all assignments for $\ca{C}$ are naturally identified with the left cosets $\Sigma_{\ca{C}}\backslash \Sigma_V$ (at this point it is nonessential, whether we consider left or right cosets, however the choice of side will be important in the following).

Let $\Cl_\Gamma$ denote the set of all possible pairs $(\ca{C},\ca{A})$ where $\ca{C}\in \ca{L}_\Gamma$ is a clustering, and $\ca{A}\in \Sigma_{\ca{C}}\backslash\Sigma_V$ is an assignment for this clustering. We introduce a partial order on the set $\Cl_\Gamma$ as follows. Notice that any relation $\ca{C}'\leq \ca{C}$ induces the inclusion of subgroups $\Sigma_{\ca{C}'}\hookrightarrow \Sigma_{\ca{C}}$ hence the natural surjection of left cosets
\[
\pr_{\ca{C}'\leq \ca{C}}\colon \Sigma_{\ca{C}'}\backslash\Sigma_V\to \Sigma_{\ca{C}}\backslash\Sigma_V.
\]
The partial order on $\Cl_\Gamma$ is defined by setting $(\ca{C}',\ca{A}')\leq (\ca{C},\ca{A})$ if and only if $\ca{C}'\leq \ca{C}$ and $\pr_{\ca{C}'\leq \ca{C}}(\ca{A}')=\ca{A}$. The poset $\Cl_\Gamma$ is graded by $\rk((\ca{C},\ca{A}))=\rk\ca{C}$.
\end{con}

\begin{defin}\label{defClusterPerm}
The poset $\Cl_\Gamma$ is called \emph{the cluster-permutohedron} of a graph $\Gamma$.
\end{defin}

\begin{con}\label{conGraphicahedron}
The construction of a graphicahedron is very similar to the previous construction, however instead of the geometric lattice $\ca{L}_\Gamma$, we use the Boolean lattice $2^{E_\Gamma}$. Let $W\subseteq E_\Gamma$ be a subset of edges of $\Gamma$. Vertex sets of connected components of the subgraph $([n],W)$ determine a clustering which we denote by $\ca{C}(W)$.

Let $\Gr_\Gamma$ denote the set of all possible pairs $(W,\ca{A})$ where $W\in 2^{E_\Gamma}$ is a set of edges, and $\ca{A}\in \Sigma_V/\Sigma_{\ca{C}(W)}$ is an assignment for the corresponding clustering $\ca{C}(W)$. The partial order on $\Gr_\Gamma$ is defined by setting $(W',\ca{A}')\leq (W,\ca{A})$ if and only if $W'\subseteq W$ and $\pr_{\ca{C}(W)'\leq \ca{C}(W)}(\ca{A}')=\ca{A}$ similar to the previous construction.
\end{con}

\begin{defin}\label{defGraphicahedron}
The poset $\Gr_\Gamma$ is called \emph{the graphicahedron} of a graph $\Gamma$.
\end{defin}

The poset $\Gr_\Gamma$ is graded by $\rk((\ca{C}(W),\ca{A}))=|W|$. However, it will be more convenient for us to forget this natural grading, and use another numerical characteristic of the elements of $\Gr_\Gamma$ that we call \emph{connectivity}. By definition,
\[
\conn((W,\ca{A}))=\rk\ca{C}(W)=n-\mbox{number of connected components of }([n],W).
\]

\begin{con}\label{conSkeleta}
By the skeleta of posets we mean the following subposets
\begin{equation}\label{eqSkeleta}
(\Cl_\Gamma)_j=\{s\in \Cl_\Gamma\mid \rk s\leqslant j\},\quad (\Gr_\Gamma)_j=\{s\in \Gr_\Gamma\mid \conn(s)\leqslant j\}.
\end{equation}
Notice that in the case of the graphicahedron we used connectivity instead of rank to define the skeleton. This definition is different from the one given in~\cite{Araujo-Pardo_Rio-Francos_Lopez-Dudet_Oliveros_Schulte:2010}. However, for the low rank elements of $\Gr_\Gamma$ there is no difference as the following lemma states.
\end{con}

\begin{lem}\label{lemConRkSimple}
If $s\in \Gr_\Gamma$ and $\rk s\leqslant 2$, then $\conn(s)=\rk s$. On the other hand, if $\rk s\geqslant 2$, then $\conn(s)\geqslant 2$.
\end{lem}

\begin{proof}
If $j=0,1$, or $2$ many edges are added to the disjoint vertex set $[n]$, then the number of connected components is reduced exactly by $j$.
\end{proof}

In particular, Lemma~\ref{lemConRkSimple} implies that $(\Gr_\Gamma)_1=\{s\in \Gr_\Gamma\mid \rk(s)\leqslant 1\}$. There is no difference between $\conn$ and $\rk$ in the definition of 1-skeleton. Let us recall the main result of~\cite{Araujo-Pardo_Rio-Francos_Lopez-Dudet_Oliveros_Schulte:2010}.

\begin{prop}\label{propGrAndCayley}
The 1-skeleton $(\Gr_\Gamma)_1$ is isomorphic to the Cayley graph $\Cayley_\Gamma$ of $\Sigma_n$ with the set of generators $\{(i, j)\mid \{i,j\}\in E_\Gamma\}$.
\end{prop}

\subsection{Toric topology}\label{subsecToric}

Let us recall several basic definitions from toric topology. For details we refer to~\cite{Buchstaber_Panov:2015}. Let $R$ denote a coefficient ring (either $\Zo$ or a field).

Let a compact torus $T$, $\dim T=k$, act on a smooth closed orientable connected manifold $X$ of real dimension $2m$. We will assume that the set $X^T$ of fixed points of the action is nonempty and finite. In this case it is said that the action has isolated fixed points.

The abelian group $N=\Hom(T,S^1)\cong \Zo^k$ is called the \emph{lattice of weights}\footnote{The term ``lattice'' appears in two different yet standard meanings in our paper}.

\begin{con}\label{conTangentWeights}
If $x\in X^T$ is an isolated fixed point, there is an induced representation of $T$ in the tangent space $T_xX$ called the tangent representation. Let $\alpha_{x,1},\ldots,\alpha_{x,m}\in \Hom(T,S^1)\cong \Zo^{k}$ be the weights of the tangent representation at $x$, which means that
\[
T_xX\cong V(\alpha_{x,1})\oplus\cdots\oplus V(\alpha_{x,m}),
\]
where $V(\alpha)$ is the standard irreducible 1-dimensional complex (or 2-dimensional real) representation given by $tz=\alpha(t)\cdot z$, $z\in \Co$. It is assumed that all weight vectors $\alpha_{x,i}$ are nonzero since otherwise $x$ is not isolated. Since there is no $T$-invariant complex structure on $X$, the vectors $\alpha_{x,i}$ are only determined up to sign. The choice of sign is nonessential for the arguments of this paper. It should be stated as well that the weight vectors are considered as the elements of rational space $N\otimes\Qo\cong\Qo^k$ if necessary.
\end{con}

\begin{defin}
If a torus $T$, $\dim T=k$, acts effectively on a manifold $X$, $\dim_\Ro X=2m$, and has isolated fixed points, the number $m-k$ is called \emph{the complexity of the action}.
\end{defin}

The slice theorem together with Construction~\ref{conTangentWeights} implies that complexity is always nonnegative.

\begin{defin}\label{defJindependent}
The $T$-action on $X$ is called \emph{$j$-independent} if, for any fixed point $x\in X^T$, any subset of $\{\alpha_{x,1},\ldots,\alpha_{x,m}\}$ of cardinality $\leqslant j$, is linearly independent (over $\Qo$).
\end{defin}

\begin{con}\label{conFacePoset}
Let us define the notion of a face of a toric action. The missing details and references can be found in~\cite{Ayzenberg_Cherepanov:2022}. For any closed connected subgroup $H\subset T$, consider the subset $X^H$ fixed by $H$. This is a closed smooth submanifold of $X$. Connected components of $X^H$ are called \emph{invariant submanifolds}. An invariant submanifold $Y$ is called a \emph{face submanifold} if it contains a $T$-fixed point. Each face submanifold is $T$-stable. Its orbit space by the $T$-action is called a \emph{face}. We denote a face by a letter $F$, while the corresponding face submanifold is denoted $X_F$, so that we have $X_F/T=F$. Let $T_F$ denote the noneffective kernel of the $T$-action on $X_F$. The dimension $\dim T/T_F$ is called the \emph{rank of $F$} (or the rank of $X_F$) and is denoted by $\rk F$.

All face submanifolds (or all faces) are ordered by inclusion. They form a finite poset graded with the rank-function. We denote this poset by $S(X)$. The poset has the greatest element $\hat{1}$ --- the manifold $X$ itself. The minimal elements are the fixed points of the action, they all have rank 0.
\end{con}

\begin{prop}[{\cite[Thm.1]{Ayzenberg_Cherepanov:2022}}]\label{propAyzCherep}
For any action with isolated fixed points, the poset $S(X)$ is locally geometric. For any fixed point $x\in X^T$, the upper order ideal $S(X)_{\geqslant x}$ is isomorphic to the lattice of flats $\Flats(\{\alpha_{x}\})$ of the linear matroid of tangent weights $\alpha_x=\{\alpha_{x,1},\ldots,\alpha_{x,m}\}\subset \Qo^k$ at $x$.
\end{prop}

\begin{con}\label{conEquivFormality}
Consider the classifying principal $T$-bundle $ET\to BT\simeq (\CP^\infty)^k$, and let $X_T=X\times_TET$ denote the Borel construction of $X$. There is a Serre fibration $p\colon X_T\stackrel{X}{\to}BT$. The cohomology ring $H^*_T(X;R)=H^*(X_T;R)$ is called the equivariant cohomology ring of $X$. Via the induced map $p^*\colon H^*(BT;R)\to H^*(X_T;R)$, the equivariant cohomology get the natural structure of the graded module over $H^*(BT;R)\cong R[k]$, the polynomial ring in $k=\dim T$ generators of degree $2$. The fibration $p$ induces Serre spectral sequence:
\begin{equation}\label{eqSerreSpSec}
E_2^{p,q}\cong H^p(BT;R)\otimes H^q(X;R)\Rightarrow H_T^{p+q}(X;R)
\end{equation}
since $BT$ is simply connected. The $T$-action on $X$ is called \emph{cohomologically equivariantly formal} (over $R$) if the Serre specral sequence~\eqref{eqSerreSpSec} collapses at $E_2$. According to~\cite[Lm.2.1]{Masuda_Panov:2006}, the action with isolated fixed points is equivariantly formal if and only if there holds $H^{\odd}(X;R)=0$.

It should be noticed that actions of algebraic torus on smooth projective complex varieties with isolated fixed points are always equivariantly formal, as follows from Bialynicki-Birula theory~\cite{Bialynicki-Birula:1973}.
\end{con}

\subsection{Isospectral matrix manifolds}\label{subsecMatrices}

In this subsection we recall several basic facts related to spaces of isospectral Hermitian matrices. Details can be found in~\cite{Ayzenberg_Buchstaber_IMRN:2021}.

\begin{con}\label{conGammaShapedMatrices}
Let $M_n$ denote the vector space of all Hermitian matrices of size $n\times n$. This is a real vector space of dimension $n^2$. Let $\Gamma=([n],E_\Gamma)$ be a simple graph. A matrix $A=(a_{i,j})\in M_n$ is called $\Gamma$-shaped if $a_{i,j}=0$ for $\{i,j\}\notin E_\Gamma$. The space $M_\Gamma$ of all $\Gamma$-shaped matrices is a real vector space of dimension $n+2|E_\Gamma|$.

For a given set $\lambda = \{\lambda_1,\ldots,\lambda_{n}\}$ of real numbers consider the subset $M_{\Gamma,\lambda}\subset M_\Gamma$ of all $\Gamma$-shaped matrices with eigenvalues $\{\lambda_1,\ldots,\lambda_{n}\}$. It follows from Sard's lemma that $M_{\Gamma,\lambda}$ is a smooth closed submanifold in $M_\Gamma$ for generic choice of $\lambda$. Simple count of parameters implies $\dim_\Ro M_{\Gamma,\lambda} = 2|E_\Gamma|$. In the following, it is always assumed that $\lambda$ is generic, in the sense that $M_{\Gamma,\lambda}$ is smooth. In most part of the paper we also assume that $\lambda_i$'s are pairwise distinct --- unless stated otherwise. Throughout the paper the term isospectral matrix manifold refers to manifolds $M_{\Gamma,\lambda}$.
\end{con}

\begin{rem}\label{remBlocks}
If $\Gamma$ is disconnected, then so is the manifold $M_{\Gamma,\lambda}$. Indeed, if $\Gamma$ has connected components $\Gamma_1,\ldots,\Gamma_s$ with $n_1,\ldots,n_s$ vertices respectively, then $\Gamma$-shaped matrix is block-shaped with blocks corresponding to $\Gamma_i$. In this case $M_{\Gamma,\lambda}$ has ${n\choose n_1,\ldots,n_s}$ pieces of the form $M_{\Gamma_1,\lambda|_{A_1}}\times\cdots\times M_{\Gamma_s,\lambda|_{A_s}}$ corresponding to the ways of distributing the eigenvalues $\lambda_i$ among the blocks. This argument is extended in the proof of Proposition~\ref{propFacePoset} below. Notice that each piece $M_{\Gamma_1,\lambda|_{A_1}}\times\cdots\times M_{\Gamma_s,\lambda|_{A_s}}$ may be disconnected itself, see Remark~\ref{remConnectivityIssue} below.
\end{rem}

\begin{con}\label{conTorusAcnOnMatrices}
Let $U(n)$ be the compact Lie group of unitary matrices and $T^{n}\subseteq U(n)$ be the compact torus of diagonal unitary matrices, $T^{n}=\{\diag(t_1,\ldots,t_{n}), t_i\in \Co, |t_i|=1\}$. The group $U(n)$ acts on $M_n$ by conjugation, which induces the action of $T^n$. In coordinate notation, if $A=(a_{i,j})$, then the $T^n$-action has the form
\begin{equation}\label{eqActionCoord}
(a_{ij})_{\substack{i=1,\ldots,n\\j=1,\ldots,n}}\mapsto
(t_it_j^{-1}a_{ij})_{\substack{i=1,\ldots,n\\j=1,\ldots,n}}
\end{equation}
It can be seen that the subspace $M_\Gamma$ is stable under the $T^n$-action. Since the conjugation preserves the spectrum, the submanifold $M_{\Gamma,\lambda}$ is stable under the $T^n$-action.
\end{con}

\begin{rem}\label{remComplexity}
Notice that $T^n$-action on $M_{\Gamma,\lambda}$ is not effective: the subgroup $\Delta(T^1)\subset T^n$ of scalar matrices acts trivially. If $\Gamma$ is connected, $\Delta(T^1)$ is precisely the noneffective kernel of the action. In general, the noneffective kernel is the product of diagonal subgroups corresponding to connected components. Hence the noneffective kernel has dimension $c(\Gamma)$, the number of connected components. Since $\dim_\Ro M_{\Gamma,\lambda}=2|E_\Gamma|$, the complexity of torus action on $M_{\Gamma,\lambda}$ equals
\begin{equation}
\compl(M_{\Gamma,\lambda})=|E_\Gamma|-(n-c(\Gamma))=\beta_1(\Gamma),
\end{equation}
the circuit rank of a graph $\Gamma$.
\end{rem}

\begin{ex}\label{exFullFlagVar}
Let $\Gamma=K_n$ be the complete graph on $n$ vertices. The manifold $M_{K_n,\lambda}$ is the manifold of all Hermitian matrices with the given eigenvalues. Notice that not only the torus, but the whole group $U(n)$ acts on $M_{K_n,\lambda}$ in this case. The $U(n)$-action is transitive, and has stabilizer $T^n$ (if $\lambda_i$ are pairwise different as assumed in Construction~\ref{conGammaShapedMatrices}). Therefore $M_{K_n,\lambda}$ is a homogeneous space $U(n)/T^n$, --- the full flag variety $\Fl(\Co^n)$. If some of $\lambda_i$'s coincide, one gets the partial flag varieties, see Section~\ref{secIdentifications}.
\end{ex}

\begin{ex}\label{exTomei}
If $\Gamma$ is a simple path $\I_n$ on $n$ vertices, having edges $\{1,2\},\ldots,\{n-1,n\}$, then $M_{\I_n,\lambda}$ is the space of tridiagonal isospectral matrices --- the Hermitian Tomei manifold. It was studied in~\cite{Tomei:1984} and~\cite{Bloch_Flaschka_Ratiu:1990}. The manifold $M_{\I_n,\lambda}$ provides an example of a quasitoric manifold which is not a toric variety, see~\cite{Davis_Januszkiewicz:1991}.
\end{ex}

\begin{ex}
The isospectral matrix spaces, corresponding to cycle-graphs, star-graphs, and indifference graphs, were studied in \cite{Ayzenberg:2020,Ayzenberg_Buchstaber_MS:2021,Ayzenberg_Buchstaber_IMRN:2021} respectively.
\end{ex}

\begin{rem}\label{remMatricesFixedPoints}
Independently of $\Gamma$, the torus action on $M_{\Gamma,\lambda}$ has exactly $n!$ isolated fixed points. Fixed points are given by the diagonal Hermitian matrices $A_\sigma=\diag(\lambda_{\sigma(1)},\ldots,\lambda_{\sigma(n)})$.
\end{rem}

\begin{prop}\label{propTangRepMatrices}
At any fixed point $A_\sigma$, the tangent representation $T_{A_\sigma}M_{\Gamma,\lambda}$ has weights of the form
\[
\{\epsilon_i-\epsilon_j\mid \{i,j\}\in\Gamma\},
\]
where $\epsilon_1,\ldots,\epsilon_n$ is the standard basis of $\Hom(T^n,S^1)\cong\Zo^n$, that is $\epsilon_i((t_1,\ldots,t_n))=t_i$ for $i=1,\ldots,n$.
\end{prop}

\begin{proof}
It is easily seen that the tangent space to $M_{\Gamma,\lambda}$ at $A_\sigma$ consists of all Hermitian $\Gamma$-shaped matrices with zeroes at the diagonal. For $i<j$ let us introduce the vector subspace $V_{i,j}\subset M_n$:
\[
V_{ij}=\{A=(a_{i',j'})\mid a_{i',j'}=0\mbox{ unless }(i',j')=(i,j)\mbox{ or }(j,i)\}
\]
(there is a unique element $a_{i,j}$ above the diagonal which is arbitrary, others are zero). This is basically the root subspace of the tangent Lie algebra of $U(n)$ corresponding to the root $\epsilon_i-\epsilon_j$ (see details in Construction~\ref{conLieGen} below). We have
\begin{equation}\label{eqDecompTangent}
T_{A_\sigma}M_{\Gamma,\lambda}=\bigoplus_{\{i,j\}\in E_\Gamma}V_{ij}.
\end{equation}
The space $V_{ij}$ is $T$-invariant, and $T$ acts on $V_{ij}$ by $t\cdot a_{i,j}\mapsto t_it_j^{-1}a_{i,j}$. Therefore $V_{ij}$ is the irreducible real 2-dimensional representation with weight $\epsilon_i-\epsilon_j$. The decomposition~\eqref{eqDecompTangent} proves the statement.
\end{proof}

\begin{rem}\label{remConnectivityIssue}
Looking at Remark~\ref{remBlocks}, it is tempting to say that whenever $\Gamma$ is connected, then so is $M_{\Gamma,\lambda}$. This holds true for all known examples. We inaccurately stated it as a general fact in~\cite{Ayzenberg_Buchstaber_IMRN:2021}, but this statement lacks the proof for the moment. Nevertheless, the modified statement holds true, which is sufficient for our considerations.
\end{rem}

\begin{lem}\label{lemConnectivity}
Let $\Gamma=([n],E_\Gamma)$ be a connected graph. Then all diagonal matrices $A_\sigma$ (that are the fixed points of the torus action) lie in a single connected component of $M_{\Gamma,\lambda}$.
\end{lem}

\begin{proof}
At first, let $\Gamma=K_2$ be a single edge. In this case $M_{K_2,\lambda}\cong \Fl(\Co^2)\cong\CP^1\cong S^2$, which is connected. Now switch to the general situation. If $\sigma=\tau\cdot(i,j)$ with $(i,j)\in E_\Gamma$, then the matrices $A_\sigma$ and $A_\tau$ are connected by a $2$-sphere of isospectral matrices of the form
\[
\begin{pmatrix}
  \lambda_{\sigma(1)} & 0 &  & \cdots &  & 0 \\
  0 & \ddots &  &  &  &  \\
   &  & * & * &  & \vdots \\
  \vdots &  & * & * &  &  \\
   &  &  &  & \ddots & 0 \\
  0 &  & \cdots &  & 0 & \lambda_{\sigma(n)}
\end{pmatrix}=\begin{pmatrix}
  \lambda_{\tau(1)} & 0 &  & \cdots &  & 0 \\
  0 & \ddots &  &  &  &  \\
   &  & * & * &  & \vdots \\
  \vdots &  & * & * &  & \vdots \\
   &  &  &  & \ddots & 0 \\
  0 &  & \cdots &  & 0 & \lambda_{\tau(n)}
\end{pmatrix}
\]
with the block on positions $i,j$. Connectivity between diagonal matrices therefore follows from the fact that any two permutations can be connected by a sequence of transpositions corresponding to edges of the connected graph $\Gamma$. This is a standard exercise: connectivity of $\Gamma$ implies connectivity of $\Cayley_\Gamma$.
\end{proof}

\subsection{Twins of isospectral matrix manifolds}\label{subsecTwins}

The notion of twin submanifolds in the full flag variety was introduced by the authors in~\cite{Ayzenberg_Buchstaber_IMRN:2021}. We briefly recall the construction and the main properties.

\begin{con}\label{conTwins}
Let $U(n)$ be the group of unitary matrices of size $n$. Consider the maximal torus $T=T^n\subset U(n)$ of diagonal matrices and two projection maps
\[
\pi_r\colon U(n)\to U(n)/T,\quad \pi_l\colon U(n)\to T\backslash U(n).
\]
Both coset spaces $U(n)/T$ and $T\backslash U(n)$ are diffeomorphic to the full flag variety $\Fl(\Co^n)$. There exists a residual action of $T$ on $U(n)/T$ from the left, and a residual action of $T$ on $T\backslash U(n)$ from the right. Both actions coincide with the standard action of a torus on a flag full variety. Both projection maps $\pi_r$ and $\pi_l$ are smooth submersions as well as fibre bundles with fiber $T$.
\end{con}

\begin{defin}\label{definTwins}
Let $X\subset U(n)/T$ be a subspace stable under the residual action of $T$ on $U(n)/T$. Notice that the full preimage $Z=\pi_r^{-1}(X)\subset U(n)$ is invariant under the two-sided action of $T\times T$ on $U(n)$. Consider the subspace $X'=\pi_l(Z)=\pi_l(\pi_r^{-1}(X))\subset T^n\backslash U(n)$. The spaces $X$ and $X'$ in the full flag variety are called the \emph{twins}.
\end{defin}

\begin{rem}\label{remSmoothnessOfTwins}
The twin $X'$ is smooth whenever $X$ is smooth, since both conditions are equivalent to smoothness of $Z=\pi_r^{-1}(X)=\pi_l^{-1}(X')$. In this case they are called twin manifolds.
\end{rem}

In~\cite{Ayzenberg_Buchstaber_IMRN:2021} we studied the manifolds of isospectral staircase Hermitian matrices and proved that their twins are regular semisimple Hessenberg varieties~\cite{Abe_Horiguchi:2019}. Even this class of examples is interesting, it shows that the twins may be non-diffeomorphic (and even non-homotopic).

Let us construct the twin of $M_{\Gamma,\lambda}$ for general graphs. As before, $\Gamma$ is a simple graph on $n$ vertices. For this graph, we denote by $\str(i)$ the neighborhood of a vertex $i$, that is the set of all vertices adjacent to $i$ united with $i$ itself.

\begin{con}\label{conTwinOfMatrixGen}
A full complex flag in $\Co^n$ can be naturally identified with the sequence of 1-dimensional linear subspaces
$L_1,\ldots,L_n\subset \Co^{n}$, which are pairwise orthogonal: $L_i\perp L_j$, $i\neq j$. Such sequence corresponds to the flag
\begin{equation}\label{eqFlagOrthogonal}
(L_1\subset L_1\oplus L_2\subset L_1\oplus L_2\oplus L_3\subset\cdots\subset\Co^n)\in\Fl(\Co^n)
\end{equation}
Consider the semisimple operator $D_\lambda\colon\Co^n\to\Co^n$ given by the diagonal matrix $\diag(\lambda_1,\ldots,\lambda_n)$ with distinct real eigenvalues. Let us define the subset $M'_{\Gamma,\lambda}\subset \Fl(\Co^n)$:
\[
M'_{\Gamma,\lambda}=\left\{\{L_i\}\in \Fl(\Co^n)\mid\forall i\in[n] \colon D_\lambda L_i\subset \bigoplus\nolimits_{j\in\str(i)} L_j\right\}.
\]
The standard action of $T^n$ on $\Co^n$ induces an action of $T^n$ on $M'_{\Gamma,\lambda}$.
\end{con}

\begin{prop}\label{propTwinsMatrices}
The manifolds $M_{\Gamma,\lambda}$ and $M'_{\Gamma,\lambda}$ are twins.
\end{prop}

The proof is completely analogous to~\cite[Thm.3.10]{Ayzenberg_Buchstaber_IMRN:2021} or~\cite[Prop.8.1]{Ayzenberg_Buchstaber_MS:2021}. We have the following general statement.

\begin{prop}\label{propOrbitsOfTwins}
If $X$, $X'$ are the twins, then the orbit spaces $X/T$ and $X'/T$ are homeomorphic, and the homeomorphism can be chosen to preserve the dimensions of toric orbits.
\end{prop}

\begin{proof}
Both $X/T$ and $X'/T$ are homeomorphic to the double quotient $T\backslash Z/T$ in the notation of Definition~\ref{definTwins}. Since each one-sided action of $T$ on $Z$ is free, an orbit of $X/T$ of dimension $j$ corresponds to an orbit of dimension $j+n$ in $Z$, which corresponds to an orbit of $X'/T$ of dimension $j$.
\end{proof}

\begin{cor}\label{corPosetsOfTwins}
The face posets of twins are isomorphic: $S(X)\cong S(X')$.
\end{cor}

\begin{rem}\label{remAboutDeMari}
The philosophy behind Construction~\ref{conTwinOfMatrixGen} is very much in spirit to the philosophy of the paper~\cite{DeMari_Shayman:1988} where (semisimple) Hessenberg varieties were first introduced. A Hessenberg variety was originally considered as a set of all unitary matrices which transform the given semisimple matrix to a matrix of a given Hessenberg form. The condition of Hessenberg form is replaced by any $\Gamma$-shaped form in our work.

The resulting space of matrices may be either a manifold $M'_{\Gamma,\lambda}$ or $M_{\Gamma,\lambda}$ depending on the choice of correspondence between unitary matrices and flags. In any case, historically, there is a strong connection between the theory of Hessenberg matrices and certain topics of computational linear algebra, such as QR-algorithms and generalized Toda flows.
\end{rem}

\section{The proof of Theorem~\ref{thmMainAll}}\label{secProofs}

With all definitions and preliminary observations of the previous section, the reader should be able to prove Theorem~\ref{thmMainAll} on one's own. However, we break the theorem into separate statements and prove them in this section.

\begin{prop}\label{propFacePoset}
For any graph $\Gamma=(V_\Gamma,E_\Gamma)$ and any generic spectrum $\lambda$, the face poset of the torus action on $M_{\Gamma,\lambda}$ is isomorphic to the cluster-permutohedron $\Cl_\Gamma$.
\end{prop}

\begin{proof}
At first we describe which subgroups $G\subset T^n$ appear as the stabilizers of the torus action on $M_{\Gamma,\lambda}$. Let $A\in M_{\Gamma,\lambda}$ be a matrix. Consider the graph $\tilde{\Gamma}=(V_\Gamma,\tilde{E})$ with edges of the form
\[
\tilde{E}=\{\{i,j\}\mid a_{i,j}\neq 0\}
\]
Obviously, $\tilde{\Gamma}\subseteq\Gamma$. Let $V_1,\ldots,V_k$ be the vertex sets of connected components of $\tilde{\Gamma}$, they define the clustering $\ca{C}=\{V_1,\ldots,V_k\}$. It follows from the coordinate form of the action~\eqref{eqActionCoord} that the stabilizer of a matrix $A$ is the subtorus
\[
T_{\ca{C}}=\Delta_{V_1}(T^1)\times\cdots\times \Delta_{V_k}(T^1),
\]
where $\Delta_{V_j}\colon T^1\to T^{V_j}$ is the diagonal map on the coordinate subtorus, corresponding to the subset $V_j$. Henceforth, we have shown that all stabilizers of the action have the form $T_{\ca{C}}$ for all possible clusterings $\ca{C}\in\Flats(\Gamma)$.

For a subgroup $T_{\ca{C}}$, consider the submanifold $M_{\Gamma,\lambda}^{T_{\ca{C}}}$ of its fixed points. It is clear that an element $A$ of this subset is a block matrix, with blocks corresponding to the subsets $V_1,\ldots,V_k$, see Remark~\ref{remBlocks}. Since $A$ lies in the isospectral set corresponding to $\lambda$, this spectrum $\lambda$ somehow distributes among blocks. Therefore the manifold $M_{\Gamma,\lambda}^{T_{\ca{C}}}$ has connected components corresponding to all possible ways of distributing the spectrum among blocks, which would give an assignment.

More formally, let $p\colon V_\Gamma\to [n]$ be a bijection, and let $X_p^{\ca{C}}$ denotes the connected component of $M_{\Gamma,\lambda}^{T_{\ca{C}}}$ which contains the diagonal matrix $\diag(\lambda_{p(1)},\ldots,\lambda_{p(1)})$. Then we have
\begin{enumerate}
  \item If $p_1=\sigma\circ p_2$, where $\sigma\in \Sigma_{V_1}\times\cdots\times\Sigma_{V_k}\subseteq\Sigma_{V_\Gamma}$, then $X_{p_1}^{\ca{C}}=X_{p_2}^{\ca{C}}$.
  \item If $p_1$ differs from $p_2$ by a permutation, that mixes blocks, then $X_{p_1}^{\ca{C}}\neq X_{p_2}^{\ca{C}}$.
\end{enumerate}

Item 1 essentially follows from Lemma~\ref{lemConnectivity} and its proof. If $p_1$ differs from $p_2$ by a transposition, lying in a single block, then the matrices $A_{p_1}$ and $A_{p_2}$ are connected by a 2-sphere. Item 1 follows from the connectivity of Cayley graphs of separate blocks.

Item 2. Assume that, on the contrary, there exist $r_1\in V_i$, $r_2\in V_j$ such that $i\neq j$ but $p_1(r_1)=p_2(r_2)=s$. Then $X_{p_1}^{\ca{C}}$ consists of matrices, in which $\lambda_s$ is the eigenvalue of the $V_i$-block, while $X_{p_2}^{\ca{C}}$ consists of matrices in which $\lambda_s$ falls in $V_j$-block. Since we assumed all eigenvalues pairwise distinct, this is a contradiction.

We recall that a face submanifold is a connected component of $M_{\Gamma,\lambda}^{T_{\ca{C}}}$ containing a fixed point. Therefore, we see that connected components of $M_{\Gamma,\lambda}^{T_{\ca{C}}}$ with this property correspond to assignments, that are the left cosets $\Sigma_{\ca{C}}\backslash\Sigma_{V_\Gamma}$. The poset of pairs (a clustering, its assignment) is exactly the cluster-permutohedron $\Cl_\Gamma$.
\end{proof}

Corollary~\ref{corPosetsOfTwins} implies

\begin{cor}\label{corTwinMatrixPoset}
The face poset $S(M'_{\Gamma,\lambda})$ of the twin submanifold of $M_{\Gamma,\lambda}$ is also isomorphic to the cluster-permutohed\-ron~$\Cl_\Gamma$.
\end{cor}

\begin{prop}\label{propClusProperties}
The following holds for the cluster-permutohedron.
\begin{enumerate}
  \item For any element $x\in\Cl_\Gamma$ of rank $0$, the upper order ideal $(\Cl_\Gamma)_{\geqslant x}$ is isomorphic to $\Flats(\Gamma)$.
  \item $\Cl_\Gamma$ is a locally geometric poset.
  \item The 1-skeleton $(\Cl_\Gamma)_1$ is isomorphic to the Cayley graph $\Cayley_\Gamma$.
\end{enumerate}
\end{prop}

\begin{proof}
All items follow directly from the above arguments. Item 1 is the consequence of Definition~\ref{defClusterPerm}. Since the maps $\pr_{\ca{C}'\leq \ca{C}}$ are surjective, the assignment $\ca{A}$ for a minimal element $x=(\ca{C},\ca{A})$ uniquely determines all assignments of larger elements. Hence $(\Cl_\Gamma)_{\geqslant x}\cong \Flats(\Gamma)_{\geqslant\hat{0}}=\Flats(\Gamma)$. Item 2 follows from item 1 by the definition of locally geometric poset.

We already noticed in the proof of Lemma~\ref{lemConnectivity} that, whenever $\sigma=\tau\cdot(i,j)$ and $\{i,j\}\in E_\Gamma$, the fixed points (=the minimal faces) $A_\sigma$ and $A_\tau$ are connected by an $T$-invariant $2$-sphere, let us denote it $\mathbb{S}_{\sigma\tau}$, and these spheres correspond to the edges of $\Cayley_\Gamma$. From a general description of all faces given in the proof of Proposition~\ref{propFacePoset}, it follows that each $\mathbb{S}_{\sigma\tau}$, as an element of face poset, has rank $1$, and any other face contains some $\mathbb{S}_{\sigma\tau}$ as a proper subset. Therefore, the 1-skeleton $(\Cl_\Gamma)_1$ is exactly the union of $n!$ fixed points, and the spheres $\mathbb{S}_{\sigma\tau}$. Hence it is isomorphic to (the face poset of) $\Cayley_\Gamma$.
\end{proof}

Now we switch to the combinatorics, and describe the relation between cluster-per\-mu\-to\-hedron $\Cl_\Gamma$ and graphicahedron $\Gr_\Gamma$.

\begin{prop}\label{propGalois}
For any graph $\Gamma$, there exists a Galois insertion $\Cl_\Gamma\hookrightarrow \Gr_\Gamma$ preserving skeleta. If $n>1$ and $j<\rk\Flats(\Gamma)=n-1$, the skeleton $(\Cl_\Gamma)_j$ is the core of $(\Gr_\Gamma)_j$.
\end{prop}

\begin{proof}
Recall that there is a Galois insertion of $\Flats(\Gamma)$ to $2^{[E_\Gamma]}$, see Example~\ref{exGaloisInsertGeomLattice}. The posets $\Cl_\Gamma$ and $\Gr_\Gamma$ are obtained from $\Flats(\Gamma)$ and $2^{[E_\Gamma]}$ respectively by attaching assignments as additional piece of information. This gives a hint on how to build a Galois connection.

We construct two maps $\iota\colon \Cl_\Gamma\rightleftarrows\Gr_\Gamma\colon \pi$. Let $(\ca{C},\ca{A})\in \Cl_\Gamma$ where $\ca{C}=\{V_1,\ldots,V_k\}$ is a clustering, and $\ca{A}$ an assignment for this clustering. Then we set
\[
\iota((\ca{C},\ca{A}))=(E_{\ca{C}},\ca{A}),
\]
where $E_{\ca{C}}$ is the union of edges of the induced subgraphs $\Gamma_{V_1},\ldots,\Gamma_{V_k}$. Conversely, if $(W,\ca{A})\in \Gr_\Gamma$, with $W\in 2^{E_\Gamma}$, then we set $\pi((W,\ca{A}))=(\ca{C}(W),\ca{A})$, where $\ca{C}$ is the clustering corresponding to $W$, see Construction~\ref{conGraphicahedron}. It is easily seen that $\iota$ and $\pi$ determine a Galois insertion. Both maps preserve skeleta (see Construction~\ref{conSkeleta}), since
\[
\conn(\iota(s))=\rk s\mbox{ and }\rk\pi(t)=\conn t
\]
for any $s\in\Cl_\Gamma$ and $t\in \Gr_\Gamma$.

Let us prove the second part of the statement. For $n=1$, the statement is trivially satisfied. Let us assume $n>1$ which guarantees both $\Cl_\Gamma\setminus\{\hat{1}\}$ and $\Gr_\Gamma\setminus\{\hat{1}\}$ are nonempty. We already proved that there is a Galois insertion of $(\Cl_\Gamma)_j$ in $(\Gr_\Gamma)_j$. In view of Lemma~\ref{lemCoreGaloisInsertion}, it only remains to prove that $(\Cl_\Gamma)_j$ is a core. Proposition~\ref{propClusProperties} implies that, for any minimal element $x\in\Cl_\Gamma$, the subset $\{s\mid s\geqslant x, \rk s\leqslant j\}\subset(\Cl_\Gamma)_j$ is isomorphic to the skeleton $\Flats(\Gamma)_j$ of the geometric lattice. Since $j<n-1$, this is a proper skeleton, so every beat contained in it is an atom; this follows from Remark~\ref{remBeatsInLattice} and Lemma~\ref{lemGeomLatticeIsCore}. Therefore, to check that $(\Cl_\Gamma)_j$ is a core, one only needs to check that its atoms are not downbeats. However each atom is an edge of the Cayley graph $\Cayley_\Gamma$, see Proposition~\ref{propClusProperties}, so it covers exactly 2 minimal elements in $\Cl_\Gamma$. This completes the proof.
\end{proof}

Now we prove the last 2 statements of Theorem~\ref{thmMainAll} regarding the girth $g$ of $\Gamma$.

\begin{prop}\label{propGirthCoincidence}
The skeleta $(\Cl_\Gamma)_{g-2}$ and $(\Gr_\Gamma)_{g-2}$ are isomorphic. In particular, when $\Gamma$ is a tree there is an isomorphism $\Cl_\Gamma\cong\Gr_\Gamma$.
\end{prop}

\begin{proof}
This a simple consequence of the definition of two posets: roughly speaking there is no difference between $2^{[E_\Gamma]}$ and $\Flats(\Gamma)$ below girth. The case of tree is obvious: if $\Gamma$ is a tree, all elements of the linear matroid $\beta(\Gamma)$ are linearly independent, so every subset of them forms a flat, and the statement is proved.

Now let us prove the statement for a general graph. Let $W\in 2^{E_\Gamma}$ and, similarly to Construction~\ref{conSkeleta}, let $\conn(W)$ equal $n$ minus the number of connected components of $([n],W)$. If $\conn W\leqslant g-2$, then pigeonhole principle implies that each connected component of $([n],W)$ contains at most $g-1$ edges. Since $g$ is the girth of $\Gamma$, each component of $([n],W)$ is a tree, and the argument of the first paragraph applies.
\end{proof}

\begin{prop}\label{propIndependency}
The torus action on $M_{\Gamma,\lambda}$ is $(g-1)$-independent.
\end{prop}

\begin{proof}
According to Proposition~\ref{propTangRepMatrices}, the weights of the tangent representation to $M_{\Gamma,\lambda}$ at any fixed point are given by the linear matroid $\beta(\Gamma)$. Assume on the contrary that there are $j\leqslant g-1$ linearly dependent vectors. Linear dependency means that the corresponding subgraph has a cycle of length $j$, which contradicts to the definition of girth~$g$.
\end{proof}

Finally, we provide one more statement, which describes the faces of $M_{\Gamma,\lambda}$ and lower ideals of $\Cl_\Gamma$.

\begin{prop}\label{propStructureOfFaces}
Let $s=(\ca{C},\ca{A})$ be an element of $\Cl_\Gamma$ with a clustering $\ca{C}=\{V_1,\ldots,V_k\}$ and an assignment $\ca{A}$ given as a class of $p\colon V_\Gamma\to [n]$.
\begin{enumerate}
  \item There is an isomorphism
  \[
  (\Cl_\Gamma)_{\leqslant s}\cong \Cl_{\Gamma|_{V_1}}\times\cdots\times \Cl_{\Gamma|_{V_k}}
  \]
  \item Let $X_s$ be a face submanifold of $M_{\Gamma,\lambda}$ corresponding to $s$. Then $X_s$ is diffeomorphic to the unique connected component of the manifold
  \[
  X_s\cong M_{\Gamma|_{V_1},\lambda_{p(V_1)}}\times\cdots\times M_{\Gamma|_{V_k},\lambda_{p(V_k)}}
  \]
  containing a diagonal matrix.
\end{enumerate}
\end{prop}

The proof follows the same lines as the proof of Proposition~\ref{propFacePoset}.

\begin{rem}
In the case of ordinary permutohedron (path graph $\I_n$), Proposition~\ref{propFacePoset} is classical: it asserts that all faces of a permutohedron are combinatorially equivalent to the products of smaller permutohedra. This observation was used by the second author in~\cite{Buchstaber:2008} to obtain differential equations on a generating series of permutohedra in the ring of polytopes. Some other series of graphs were studied in~\cite{Buchstaber:2008}, but the focus of that paper was made on graph-associahedra, which are convex polyhedra. These are different from cluster-permutohedra (with the permutohedra being the only exception). However, a similar technique of generating series can be applied to study combinatorics of cluster-permutohedra. This is the subject of a different paper.
\end{rem}

\section{Results on graphicahedra are applied in toric topology}\label{secNonFormality}

We recall two particular examples of graphicahedra/cluster-permutohedra corresponding to the star-graph $\St_3$ and to cycle graphs $\Cy_n$. They were studied in~\cite{Rio-Francos_Hubard_Oliveros_Schulte:2012} and independently in~\cite{Ayzenberg_Buchstaber_MS:2021} and~\cite{Ayzenberg:2020}. In the case of $\Cy_n$, the related construction was independently introduced by Panina~\cite{Panina:2015} under the name \emph{cyclopermutohedron}. So far, the topology of graphicahedra in these two cases is well-studied. The aim of this section is to prove that the known results about these posets imply non-equivariant formality of the corresponding isospectral matrix manifolds $M_{\St_3,\lambda}$ and $M_{\Cy_n,\lambda}$ for $n\geqslant 4$.

\begin{ex}
Let $\St_3=K_{1,3}$ be the star graph with 3 rays. Since $\St_3$ is a tree, Proposition~\ref{propGirthCoincidence} implies the isomorphism of posets $\Cl_{\St_3}\cong \Gr_{\St_3}$. These posets have rank $3$, this coincides with the dimension of effectively acting torus $T^3$ on a $6$-dimensional manifold $M_{\St_3,\lambda}$. The torus action has complexity zero, in accordance with Remark~\ref{remComplexity}. Proposition~\ref{propClusProperties} implies that for each minimal element $x\in\Cl_{\St_3}$, the poset $(\Cl_{\St_3})_{\geqslant x}$ is a boolean lattice of rank $3$. In particular, reversing the order in $\Cl_{\St_3}$ we get a simplicial poset of dimension 2 (actually it is even a balanced simplicial complex).
\end{ex}

\begin{prop}[\cite{Rio-Francos_Hubard_Oliveros_Schulte:2012} and~\cite{Ayzenberg_Buchstaber_MS:2021}]\label{propStarProps}
The poset $\Cl_{\St_3}\setminus\{\hat{1}\}$ is isomorphic to the face poset of a certain regular cell subdivision of $T^2$ into hexagons.
\end{prop}

\begin{ex}
Let $\Cy_n$ be the cycle graph with $n$ vertices. The torus action on $M_{\Cy_n,\lambda}$ has complexity $1$, and this action is $(n-1)$-independent, since girth of $\Cy_n$ equals $n$, see Proposition~\ref{propIndependency}. Torus actions with these properties were studied by the first author in~\cite{Ayzenberg:2018} under the name torus actions of complexity one in general position. A prominent role in the theory of such actions is played by the face poset, and the union of all proper faces, which was called a \emph{sponge}.
\end{ex}

In~\cite{Rio-Francos_Hubard_Oliveros_Schulte:2012} and, independently, in~\cite{Ayzenberg:2020} the following construction appeared.

\begin{con}
Let $\alpha_1,\ldots,\alpha_n\in\Ro^{n-1}$ be a collection of vectors having equal lengths and equal pairwise angles. We have $\sum_{i=1}^{n}\alpha_i=0$. Let $N\subset\Ro^{n-1}$ be the lattice generated by $\alpha_i$'s, and let $N'\subset N$ be the sublattice generated by $\alpha_1-\alpha_2,\ldots,\alpha_{n-1}-\alpha_n$. There holds $N/N'\cong\Zo_n$.

The Voronoi subdivision of $\Ro^{n-1}$ generated by $N$ is the regular tiling of $\Ro^{n-1}$ by parallel copies of a permutohedron. Taking quotient of this subdivision by $N'$ we get a regular cell subdivision of a torus $T^{n-1}=\Ro^{n-1}/N'$, which has $n$ maximal cells, each isomorphic to a permutohedron. In~\cite{Ayzenberg:2020} the first author called the resulting cell subdivision \emph{the wonderful subdivision of a torus} and denoted the corresponding poset by $\PT^{n-1}$.
\end{con}

\begin{prop}[\cite{Rio-Francos_Hubard_Oliveros_Schulte:2012},\cite{Ayzenberg:2020}]\label{propCycleProps}
For $n\geqslant 3$, the poset $\Cl_{\Cy_n}\setminus\{\hat{1}\}$ is isomorphic to the wonderful subdivision $\PT^{n-1}$ of a torus.
\end{prop}

The next statement follows from Propositions~\ref{propStarProps} and~\ref{propCycleProps}.

\begin{cor}\label{corNonAcyclic}
If $\Gamma$ is either $\St_3$ or $\Cy_n$, $n\geqslant4$, we have $H_1(|(\Cl_\Gamma)_2|;\Zo)\neq 0$.
\end{cor}

In~\cite{Ayzenberg_Masuda_Solomadin:2022} the following result was proved.

\begin{prop}[\cite{Ayzenberg_Masuda_Solomadin:2022}]\label{propSkeletonAcyclic}
If an action of $T$ on $X$ is equivariantly formal and $j$-independent, then the poset $S(X)_r=\{F\in S(X)\mid \rk F\leqslant r\}$ is $\min(\dim S(X)_r-1,j+1)$-acyclic for any $r>0$.
\end{prop}

The phrase ``a poset $S$ is $t$-acyclic'' means $\Hr_i(|S|;\Zo)=0$ for $i\leqslant t$. Corollary~\ref{corNonAcyclic} and Proposition~\ref{propSkeletonAcyclic} imply the following statement.

\begin{prop}\label{propStCyNonFormal}
The torus actions on the manifolds $M_{\St_3,\lambda}$ and $M_{\Cy_n,\lambda}$, $n\geqslant 4$, are not equivariantly formal.
\end{prop}

This result was proved in~\cite{Ayzenberg_Buchstaber_MS:2021} and~\cite{Ayzenberg:2020}, but instead of Proposition~\ref{propSkeletonAcyclic} slightly different arguments were used.

\begin{rem}
Proposition~\ref{propStCyNonFormal} provides examples of non-formal isospectral matrix spaces. However, there also exist examples of equivariantly formal manifolds, see Example~\ref{exFullFlagVar} and~\ref{exTomei}. If $\Gamma$ is the complete graph $K_n$, the manifold $M_{K_n,\lambda}$ is the full flag variety $\Fl(\Co^n)$. It has even-dimensional Bruhat cell decomposition, hence $H^{\odd}(\Fl(\Co^n);\Zo)=0$, hence $M_{K_n,\lambda}$ is equivariantly formal. If $\Gamma$ is the path graph $\I_n$, the Hermitian Tomei manifold $M_{\I_n,\lambda}$ is a quasitoric manifold over a permutohedron, and all quasitoric manifolds are equivariantly formal. More generally, in~\cite{Ayzenberg_Buchstaber_IMRN:2021} we proved that the manifolds of isospectral staircase Hermitian matrices (which correspond to indifference graphs) are all equivariantly formal. Both $K_n$ and $\I_n$ are indifference graphs.
\end{rem}

The discussion above motivates the following problem.

\begin{probl}\label{probMainFormality}
For which graphs $\Gamma$, the torus action on a generic manifold $M_{\Gamma,\lambda}$ is cohomologically equivariantly formal?
\end{probl}

A similar problem can be posed for the discrete 2-torus actions on the manifold $M^{\Ro}_{\Gamma,\lambda}$ of real symmetric $\Gamma$-shaped matrices with the given spectrum $\lambda$, and equivariant formality over $\Zt$. Moreover, a problem similar to Problem~\ref{probMainFormality} can be posed in all Lie types in the following sense.

\begin{con}\label{conLieGen}
Recall that $M_n$ denotes the set of all Hermitian matrices. Multiplying all matrices by $\sqrt{-1}$ we get skew-Hermitian matrices, so we may identify $M_n$ with the Lie algebra $\uuu(n)$ of type $A$. The action of the compact Lie group $U(n)$ on $M_n$ is identified with the (co)adjoint action of $U(n)$, and $T^n\subset U(n)$ is considered as a maximal torus.

In a much similar way, one can take the (co)adjoint representation of any semisimple compact Lie group $G$ on the tangent algebra $\g$. Let $E\subseteq \Phi^+$ be a subset of positive roots and
\[
M_E=\ttt\oplus\bigoplus_{\alpha\in E}L_\alpha
\]
the corresponding $T$-subrepresentation in $\g$. Consider the intersection of $M_E$ with a generic $G$-orbit in $\g$:
\[
M_{E,\lambda}=M_E\cap G\lambda.
\]
It can be seen that there is an induced $T$-action on $M_{E,\lambda}$. The space $M_{\Gamma,\lambda}$ appears as a particular case if one works with Lie type A. A general problem emerges.
\end{con}

\begin{probl}\label{probLieTypesGen}
For which positive root sets $E\subseteq\Phi^+$, the manifold $M_{E,\lambda}$ (for generic $\lambda$) is equivariantly formal?
\end{probl}

Looking at Construction~\ref{conLieGen}, one concludes that there is a meaningful notion of cluster-permutohedra of general Lie types. At least, they can be defined synthetically by
\begin{equation}\label{eqLieTypeClusterPerm}
\Cl_{E}=S(M_{E,\lambda}),
\end{equation}
as the face poset of the torus action on $M_{E,\lambda}$. We believe that such cluster-permutohedra deserve further study in relation to Problem~\ref{probLieTypesGen}.

\section{Identification constructions for posets}\label{secIdentifications}

\subsection{Identification constructions}

The aim of this section is to describe the fundamental similarity between cluster-permutohedra, graphicahedra and quasitoric manifolds/small covers. We introduce a construction of a symmetric identification set. This construction provided with appropriate structures gives all the above-listed objects. A little modification of this construction allows to describe face posets of Grassmann manifolds, manifolds of partial flags and many other related manifolds with the torus action.

\begin{con}\label{conIdentSetBasic}
Let $M$ be a set, $G$ be a group and $\sbgrps(G)$ the set of all subgroups of $G$, and $Z$ a set with a left action of $G$. A map $\Lambda\colon M\to\sbgrps(G)$ will be called a characteristic map. The \emph{symmetric identification set} is defined as the identification set
\[
M\times_\Lambda Z=(M\times Z)/\simc,
\]
where $(m_1,z_1)\sim(m_2,z_2)$ if and only if $m_1=m_2$ and $z_1,z_2$ lie in the same $\Lambda(m_1)$-orbit on~$Z$.
\end{con}

The most common use of this construction is the situation when $Z=G$ and $G$ acts on $G$ by multiplication from left. In this case
\[
M\times_\Lambda G=(M\times G)/\simc,
\]
where $(m_1,g_1)\sim(m_2,g_2)$ if and only if $m_1=m_2$ and $g_1g_2^{-1}\in\Lambda(m_1)$. In this case we have an induced right $G$-action on $M\times_\Lambda G$ given by: $[(m,g_1)]\cdot g=[(m,g_1g)]$, and the quotient $(M\times_\Lambda G)/G$ is canonically bijective to $M$.

Generally, if $\Lambda$ takes values in normal subgroups, a symmetric identification set $M\times_\Lambda Z$ attains the well-defined left $G$-action given by $g\cdot[(m,z)]=[(m,gz)]$. In this case, the orbit set $(M\times_\Lambda Z)/G$ is isomorphic to $M\times(Z/G)$.

There are obvious enrichments of this construction to common categories.

\begin{con}\label{conIdentSetTop}
We may pose topology on all objects under consideration. In this case, $M$ and $Z$ are topological spaces, $G$ is a topological group, $\sbgrps(G)$ is the set of all closed subgroups of $G$ ordered by inclusion, and a characteristic map $\Lambda\colon M\to\sbgrps(G)$ is usually assumed upper semicontinuous, see~\cite{Buchstaber_Ray:2008}. Then the symmetric identification set $M\times_\Lambda Z$ gets the standard identification topology, the $G$-action on this space is continuous and the quotient is homeomorphic to $M\times (Z/G)$.
\end{con}

\begin{con}\label{conIdentSetPoset}
Let $M$ be a poset, $\Lambda\colon M\to\sbgrps(G)$ be monotonic, and $Z$ be a trivial poset (every two distinct elements are incomparable). Then the symmetric identification set $M\times_\Lambda Z$ gets the partial order induced from $M$:
\[
[(m_1,z_1)]\leqslant[(m_2,z_2)] \mbox{ if and only if }m_1\leqslant m_2\mbox{ and }z_1,z_2\mbox{ lie in }\Lambda(m_2)\mbox{-orbit}.
\]
In this case $M\times_\Lambda Z$ is a poset. If $M$ is graded, then so is $M\times_\Lambda Z$. The $G$-action on $M\times_\Lambda Z$ (if it exists) preserves the order, and its quotient is isomorphic to $M\times (Z/G)$ as a poset.
\end{con}

Topological and ordered settings are connected to each other in all meaningful ways. For example, Construction~\ref{conIdentSetPoset} is a particular case of Construction~\ref{conIdentSetTop} if one replaces the poset by the corresponding Alexandroff topology.

\begin{ex}\label{exSmallCover}
Let $P$, $\dim P=n$, be a simple convex polytope with facets $\Facets(P)=\{\F_1,\ldots,\F_m\}$, let $G=\Zt^n$, and let $\lambda\colon\Facets(P)\to \Zt^n$ be a map satisfying the following condition: if $F=\F_{i_1}\cap\cdots\cap\F_{i_k}$ is a nonempty face of $P$, then $\lambda(\F_{i_1}),\ldots,\lambda(\F_{i_k})$ are linearly independent in $\Zt^n$. Then we can define a map $\Lambda\colon P\to \sbgrps(\Zt^n)$ by setting
\[
\Lambda(x)=\lambda(\F_{i_1})\times\cdots\times\lambda(\F_{i_k})\mbox{ if } x \mbox{ lies in the interior of }\F_{i_1}\cap\cdots\cap\F_{i_k}.
\]
In this case, the symmetric identification space $P\times_\Lambda\Zt^n$ with the action of $\Zt^n$ is called a \emph{small cover} over $P$.

In a similar fashion, a quasitoric manifold $P\times_\Lambda T^n$ is defined, see~\cite{Davis_Januszkiewicz:1991,Buchstaber_Panov:2015}. Moment-angle manifolds and their real analogues are defined similarly.
\end{ex}

\begin{ex}\label{exClusterIdent}
Now recall that $\Flats(\Gamma)$ denotes the poset of flats of the graphical matroid of a simple graph $\Gamma=([n],E_\Gamma)$. The elements of $\Flats(\Gamma)$ can be identified with clusterings $\ca{C}=\{V_1,\ldots,V_k\}$, see Construction~\ref{conClusterPerm}. Let $G=Z=\Sigma_n$ and let us define $\Lambda\colon \Flats(\Gamma)\to \sbgrps(\Sigma_n)$ by $\Lambda(\ca{C})=\Sigma_{\ca{C}}=\Sigma_{V_1}\times\cdots\times\Sigma_{V_k}$. Then the symmetric identification poset
\[
\Flats(\Gamma)\times_\Lambda\Sigma_n
\]
is isomorphic to the cluster-permutohedron $\Cl_\Gamma$, simply by the definition.

In a similar way, the graphicahedron $\Gr_\Gamma$ coincides with the symmetric identification poset $2^{E_\Gamma}\times_\Lambda\Sigma_n$.
\end{ex}

\subsection{Partial flags and beyond}

We now switch to more complicated examples. In~\cite{Ayzenberg_Cherepanov:2022} the question was posed: how one can describe the face poset of Grassmann manifolds and more general manifolds of partial flags with arbitrary torus actions in combinatorial terms? Here we answer this question for the standard torus action on manifolds of partial flags.

\begin{con}
Let $k_1,\ldots,k_s$ be a sequence of positive integers, $n=\sum_{1}^{s}k_i$. This defines an increasing sequence $(d_1,\ldots,d_s)$, $0<d_1<\cdots<d_{s-1}<d_s=n$ where $d_j=\sum_{1}^{j}k_i$. We have a manifold of partial flags (of type A)
corresponding to $\bar{k}=(k_1,\ldots,k_s)$:
\[
\Fl_{\bd}=\{V_\bullet=\{0\}\subset V_{d_1}\subset \cdots\subset V_{d_{s-1}}\subset \Co^n\mid \dim_{\Co}V_{d_i}=d_i\}.
\]
It is diffeomorphic to $U(n)/(U(k_1)\times\cdots\times U(k_s))$.
The standard $T^n$-action on $\Co^n$ induces the $T^n$-action on $\Fl_{\bd}$ with 1-dimensional non-effective kernel $\Delta(T^1)$. 
\end{con}

The only $T$-stable subspaces are the coordinate spaces, which implies

\begin{lem}\label{lemFixedPointPartFlags}
Fixed points of the $T$-action on $\Fl_{\bd}$ are the coordinate flags
\[
A_{J_1,\ldots,J_s}=\Co^{J_1}\subset \Co^{J_1\sqcup J_2}\subset\cdots\subset \Co^{J_1\sqcup \cdots\sqcup J_s}\mbox{ with }\bigsqcup J_i=[n].
\]
There are
\begin{equation}\label{eqNumberFixedPtsPartFlags}
|\Fl_{\bd}^{T}|=\dfrac{n!}{k_1!k_2!\cdots k_s!}
\end{equation}
many isolated fixed points.
\end{lem}

\begin{prop}\label{propPartFlagConnStab}
For a $T$-action on $\Fl_{\bd}$ the stabilizers $T_x$ have the form $\Delta_{J_1}(T^1)\times\cdots\times\Delta_{J_s}(T^1)$ where $|J_i|=k_i$, $\{J_1,\ldots,J_s\}$ is a partition of $[n]$, and $\Delta_{J_i}\colon T^1\to T^{J_i}$ is the diagonal map into the coordinate subtorus. The tangent weights at a fixed point $A_{J_1,\ldots,J_s}$ are given by
\begin{equation}\label{eqRootsMultiPart}
\{e_i-e_j\mid i\in J_a, j\in J_b, a<b\}
\end{equation}
where $e_1,\ldots,e_n$ is the standard basis of $\Hom(T,S^1)\cong\Zo^n$.
\end{prop}

\begin{proof}
We give a sketch of proof. Since the action on $\Fl_{\bd}$ can be extended to the algebraic torus actions and $\Fl_{\bd}$ is smooth projective, one can apply Bialynicki-Birula theory~\cite{Bialynicki-Birula:1973}. For a given point $x$, the limiting point $l(x)=\lim_{t\to 0}tx$ is a fixed point. Notice that there holds $T_{tx}=T_x$. So far, to prove the proposition one only needs to find stabilizers in vicinity of fixed points, that is to look at all tangent representations. The analysis of tangent representations is pretty much similar to Proposition~\ref{propTangRepMatrices} and left as an exercise to the reader.
\end{proof}

Since the type A positive roots listed in~\eqref{eqRootsMultiPart} correspond to edges of the complete multipartite graph $K_{J_1,\ldots,J_s}$ with parts of sizes $|J_i|=k_i$, Proposition~\ref{propPartFlagConnStab} together with Proposition~\ref{propAyzCherep} imply the following statement.

\begin{cor}\label{corMatroidPartFlags}
For any fixed point $x=A_{J_1,\ldots,J_s}$, the poset $S(\Fl_{\bd})_{\geqslant x}$ is isomorphic to $\Flats(K_{k_1,\ldots,k_s})$ --- the geometric lattice of the graphical matroid corresponding to the complete multipartite graph $K_{k_1,\ldots,k_s}$.
\end{cor}

\begin{rem}\label{remFaceOfPartFlags}
Notice that an element of $\Flats(K_{k_1,\ldots,k_s})$ is a clustering $\ca{C}=\{V_1,V_2,\ldots\}$ of a complete multipartite graph. Each induced subgraph $\Flats(K_{k_1,\ldots,k_s})|_{V_i}$ is a complete multipartite graph itself. The face submanifold corresponding to such element $\Ca{C}$ is diffeomorphic to the product of partial flag manifolds corresponding to all clusters.
\end{rem}

Now we introduce the analogue of cluster-permutohedron for partial flags, cf. Example~\ref{exClusterIdent}.

\begin{con}\label{conPartFlagsPoset}
Let ${[n]\choose k_1,\ldots,k_s}$ denote the set of all ordered subdivisions of $[n]$ into labelled parts of sizes $k_1,\ldots,k_s$. The set ${[n]\choose k_1,\ldots,k_s}$ is identified with the right coset $\cong\Sigma_n/(\Sigma_{k_1}\times\cdots\times\Sigma_{k_s})$ of the symmetric group. The symmetric group $\Sigma_n$ naturally acts on ${[n]\choose k_1,\ldots,k_s}$. Consider the poset $S=\Flats(K_{k_1,\ldots,k_s})$ and the characteristic map $\Lambda\colon S\to \sbgrps(\Sigma_n)$, $\Lambda(\ca{C})=\Sigma_{\ca{C}}$, as in Example~\ref{exClusterIdent}. Then we get the symmetric identification poset
\begin{equation}\label{eqPartFlagsPoset}
\Flats(K_{k_1,\ldots,k_s})\times_\Lambda {[n]\choose k_1,\ldots,k_s}.
\end{equation}
\end{con}

\begin{prop}\label{propMatrMultiplicities}
The poset of faces of the partial flag manifold $\Fl_{\bd}$ is isomorphic to the poset~\eqref{eqPartFlagsPoset}.
\end{prop}

\begin{proof}
The manifold $\Fl_{\bd}$ is identified with the manifold of all Hermitian matrices having eigenvalues $\lambda_1,\ldots,\lambda_s$ with multiplicities $k_1,\ldots,k_s$ respectively. With this done, the proof becomes similar to the proof of Proposition~\ref{propClusProperties}.
\end{proof}

A more general statement and more detailed proof are given in Theorem~\ref{thmMegaDescription} below. We now combine the constructions of cluster-permutohedra and partial flag variety.

\begin{con}
Let $k_1,\ldots,k_s$ be a sequence defining a complete bipartite graph $K_{k_1,\ldots,k_s}$. For a vertex $v$ of this graph let $\pa(v)$ denote the index of its part. Let $\Gamma=([s],E_\Gamma)$ be any graph on the set $[s]$ of parts. Let $\Gamma(k_1,\ldots,k_s)$ denote the subgraph of $K_{k_1,\ldots,k_s}$ which contains an edge $\{v,w\}$ if and only if $\{\pa(v),\pa(w)\}\in E_\Gamma$. Using the terminology of~\cite{Bjorner_Wachs_Welker:2005} $\Gamma(k_1,\ldots,k_s)$ is the inflation of $\Gamma$, where the $i$-th vertex is replaced by $k_i$ copies. Example is shown on Fig.~\ref{figInflated}, A. If, in addition, all edges within each part of $\Gamma(k_1,\ldots,k_s)$ are added, the resulting graph is denoted by $\tilde{\Gamma}(k_1,\ldots,k_s)$


A partial flag variety $\Fl_{\bd}$ determined by $k_1,\ldots,k_s$ may be identified with the space of pairwise orthogonal vector subspaces $L_1,\ldots,L_s$ of dimensions $k_1,\ldots,k_s$ respectively. Consider the subspace
\[
\Mbd'_{\Gamma,\lambda}=\left\{\{L_i\}\in \Fl_{\bd}\mid\forall i\in[s] \colon D_\lambda L_i\subset \bigoplus\nolimits_{j\in\str(i)} L_j\right\},
\]
where $D_\lambda$ is a diagonal matrix, cf. Construction~\ref{conTwinOfMatrixGen}. The space $\Mbd'_{\Gamma,\lambda}$ is stable under the $T$-action on $\Fl_{\bd}$ since the torus of diagonal matrices commutes with $D_\lambda$.
\end{con}

\begin{thm}\label{thmMegaDescription}
For generic $\lambda$, the space $\Mbd'_{\Gamma,\lambda}$ is a smooth submanifold. The face poset $S(\Mbd'_{\Gamma,\lambda})$ is isomorphic to
\begin{equation}\label{eqPartGraphPoset}
\Flats(\Gamma(k_1,\ldots,k_s))\times_\Lambda {[n]\choose k_1,\ldots,k_s},
\end{equation}
with characteristic function $\Lambda$ defined as in Construction~\ref{conPartFlagsPoset}.
\end{thm}

\begin{proof}
Let us prove smoothness. Notice that there is a smooth fibration $p\colon\Fl(\Co^n)\to\Fl_{\bd}$, which forgets some intermediate subspaces in a full flag. If a flag is represented as a collection of orthogonal lines (see~\eqref{eqFlagOrthogonal}), the map $p$ operates as follows
\[
p\colon (L_1,\ldots,L_n)\mapsto \left(\bigoplus_{i=1}^{k_1}L_i,\bigoplus_{i=k_1+1}^{k_2}L_i,\ldots,\bigoplus_{i=n-k_s+1}^{n}L_i\right)
\]
It is easily seen that the preimage of the subspace $\Mbd'_{\Gamma,\lambda}$ under $p$ coincides with
\[
M'_{\tilde{\Gamma}(k_1,\ldots,k_s),\lambda}
\]
--- the twin of the isospectral manifold, see Construction~\ref{conTwinOfMatrixGen} and Proposition~\ref{propTwinsMatrices}. The isospectral space $M_{\tilde{\Gamma}(k_1,\ldots,k_s),\lambda}$ is smooth for generic $\lambda$, see Construction~\ref{conGammaShapedMatrices}. Therefore its twin $M'_{\tilde{\Gamma}(k_1,\ldots,k_s),\lambda}$ is also smooth for the same $\lambda$ according to Remark~\ref{remSmoothnessOfTwins}. Now the smoothness of $\Mbd'_{\Gamma,\lambda}=p(M'_{\tilde{\Gamma}(k_1,\ldots,k_s),\lambda})$ is proved by the following

\begin{clai}
If $p\colon X\to Y$ is a smooth fibration (hence a submersion) and $A\subset Y$ is a subspace such that $p^{-1}(A)$ is a smooth submanifold in $X$, then $A$ is a smooth submanifold in $Y$.
\end{clai}

The claim follows easily from the submersion theorem. It proves the smoothness of $\Mbd'_{\Gamma,\lambda}$ for generic $\lambda$.

Now let us prove the statement about its face poset.

\textbf{I.} First notice that $T$-fixed points of both $\Mbd'_{\Gamma,\lambda}$ and $\Fl_{\bd}$ are represented by coordinate partial flags of type $\bd$, they can be identified with the set ${[n]\choose k_1,\ldots,k_s}$. In other words, the minimal elements of both $S(\Mbd'_{\Gamma,\lambda})$ and $S(\Fl_{\bd})$ correspond to the ways of putting the balls numbered $1,2,\ldots,n$ in the boxes of sizes $k_1,\ldots,k_s$. It is allowed to mix balls in each box.

\textbf{II.} The edges of $S(\Fl_{\bd})$ have the following description, which generalizes the way the Johnson graph appears as the GKM-graph of a Grassmann manifold (see~\cite[\S 9.2]{Guilleminn_Zara:2001}). Two distributions of balls are connected by an edge in $S(\Fl_{\bd})$, if one is obtained from another by exchanging a pair of balls from different boxes. If we exchanged the balls $i$ and $j$, then the weight of the corresponding edge equals $\pm(\epsilon_i-\epsilon_j)$ where $\epsilon_1,\ldots,\epsilon_n$ is the basis of the weight lattice. See Figure~\ref{figGrassmann} for an example $k_1=k_2=2$, which corresponds to the Grassmann manifold $\Gr_{4,2}$.

\begin{figure}[h]
\begin{center}
\includegraphics[scale=0.4]{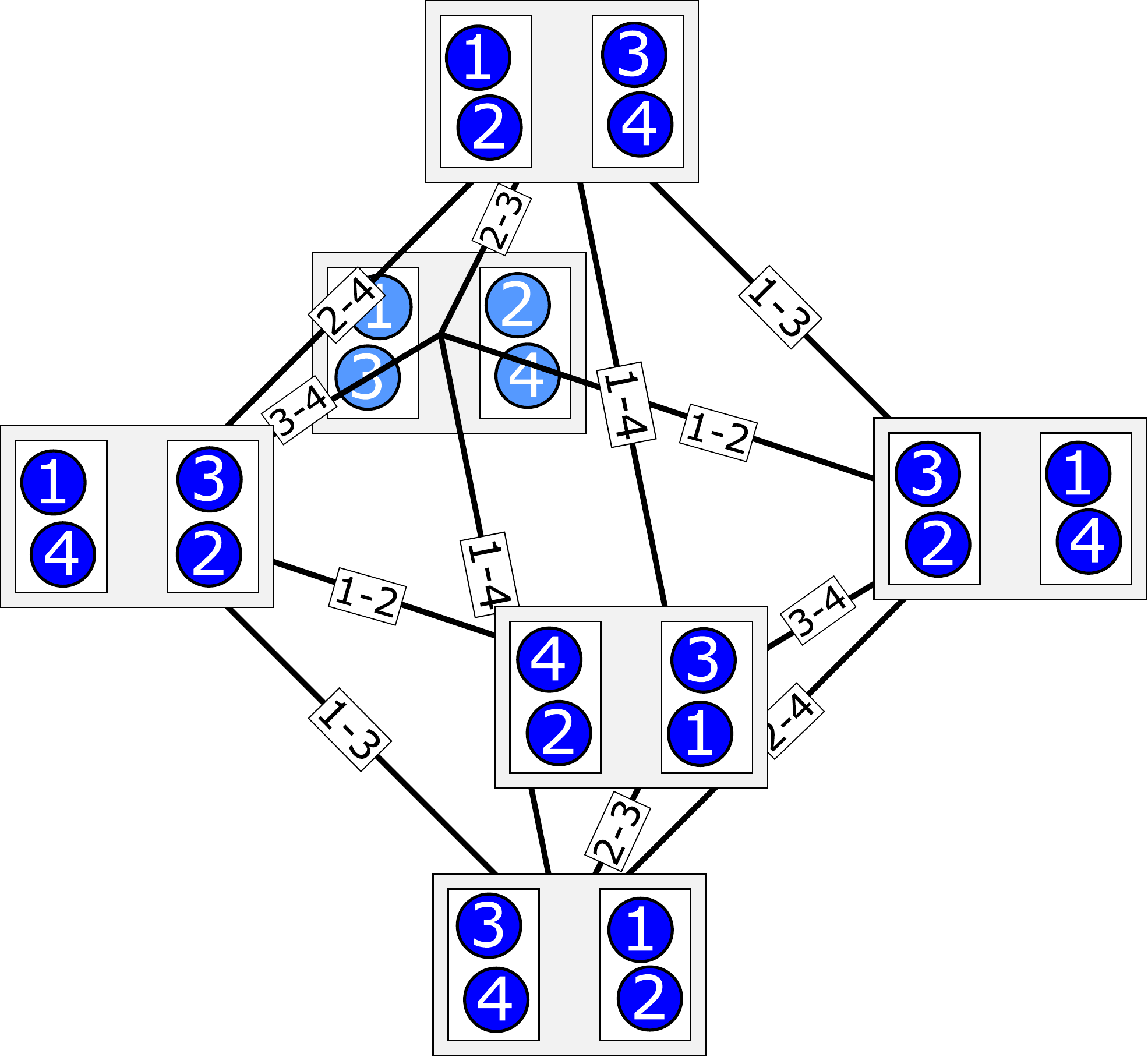}
\end{center}
\caption{GKM graph of the Grassmann manifold $\Gr_{4,2}$.}\label{figGrassmann}
\end{figure}

There exists an edge between two balls' distributions in $S(\Mbd'_{\Gamma,\lambda})$ if and only if we exchange the balls from two boxes, whose indices form an edge in $\Gamma$. More generally, if we have a clustering

Let $x\in S(\Mbd'_{\Gamma,\lambda})$ be a fixed point encoded by balls' distribution $(J_1,\ldots,J_s)$ which means that $i$-th box contains the subset of balls $J_i\subset[n]$, $|J_i|=k_i$. Then the weights of edges incident to $x$ in $S(\Mbd'_{\Gamma,\lambda})$ have the form
\[
\{\pm(\epsilon_r-\epsilon_s)\mid r\in J_i, s\in J_j,\mbox{ and }\{i,j\}\in E_\Gamma\}.
\]
Therefore the local poset $S(\Mbd'_{\Gamma,\lambda})_{\geqslant x}$ is isomorphic to the geometric lattice of the graphical matroid corresponding to $\Gamma(k_1,\ldots,k_s)$.

\textbf{III.} With the preliminary work done, let us describe the isomorphism between $\ca{P}=\Flats(\Gamma(k_1,\ldots,k_s))\times_\Lambda {[n]\choose k_1,\ldots,k_s}$ and $\ca{F}=S(\Mbd'_{\Gamma,\lambda})$.

Notice, that an element of $\ca{P}$ is a pair $(\ca{C},\kappa)$, where $\ca{C}=(V_1,\ldots,V_l)$ is a clustering of $\Gamma(k_1,\ldots,k_s)$ (i.e. a flat of the corresponding graphical matroid), and $\kappa\in \Sigma_{\ca{C}}\backslash \Sigma_n/(\Sigma_{k_1}\times\cdots\times\Sigma_{k_s})$ is a bicoset. Indeed, the set ${[n]\choose k_1,\ldots,k_s}$ is isomorphic to $\Sigma_n/(\Sigma_{k_1}\times\cdots\times\Sigma_{k_s})$, while the additional quotient from the left arises from the characteristic function $\Lambda$. Now we construct the face submanifold of $\Mbd'_{\Gamma,\lambda}$ corresponding to $(\ca{C},\kappa)$.

As before, if $v$ is a vertex of $\Gamma(k_1,\ldots,k_s)$, let $\pa(v)\in[s]$ denote the index of the part where $v$ lies (i.e. the vertex of $\Gamma$ whose inflation produces $v$). Any clustering $\ca{C}=(V_1,\ldots,V_l)$ of $\Gamma(k_1,\ldots,k_s)$ produces the set of numbers $k_i(V_j)=\#\{v\in V_j\mid \pa(v)=i\}$ for $i\in[s]$, $j=1,\ldots,l$. Notice that these numbers are invariants of the right $\Sigma_{k_1}\times\cdots\times\Sigma_{k_s}$-action on $\Flats(\Gamma(k_1,\ldots,k_s))$: if one permutes objects inside each box, their number does not change. See example shown on Fig.~\ref{figInflated}, B.

\begin{figure}[h]
\begin{center}
\includegraphics[scale=0.32]{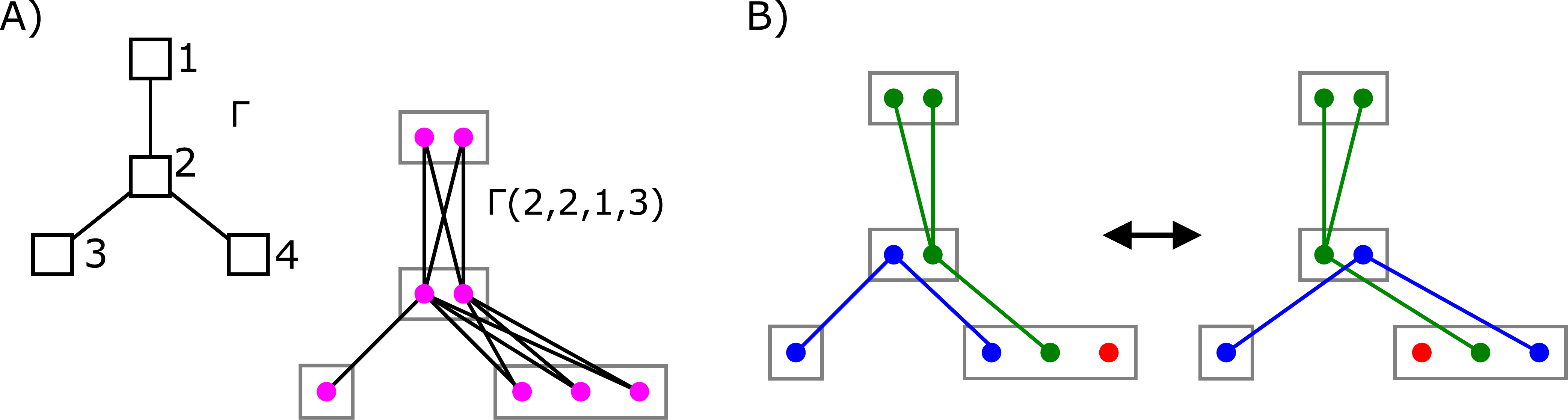}
\end{center}
\caption{(A) An example of a vertex-inflation of a graph. (B) An example of a clustering in an inflated graph. 
}\label{figInflated}
\end{figure}

Let $\kappa$ be the bicoset of $\sigma\in\Sigma_n$. For each cluster $V_j$, $j=1,\ldots,l$, consider the coordinate subspace $\Co^{\sigma(V_j)}$ of $\Co^n$ and the manifold
\[
\prescript{\overline{k(V_j)}}{\sigma}{M'_{\Gamma,\lambda}}=\{(L_{1,j},\ldots,L_{s,j})\},
\]
where $L_{i,j}$ are mutually orthogonal subspaces of $\Co^{\sigma(V_j)}$, $\dim_\Co L_{i,j}=k_i(V_j)$, and $D_\lambda L_{i,j}\subseteq \bigoplus_{i'\in\str(i)}L_{i',j}$. Finally, consider the submanifold of $\Mbd'_{\Gamma,\lambda}$ defined as the connected component of
\[
X_{(\ca{C},\kappa)}=\left\{(L_1,\ldots,L_s)\mid \forall i\colon L_i=\bigoplus_{j=1}^l L_{i,j},\mbox{ where }\forall j\colon (L_{i,j})\in \prescript{\overline{k(V_j)}}{\sigma}{M'_{\Gamma,\lambda}}\right\}.
\]
containing a fixed point. It can be seen that $X_{(\ca{C},\kappa)}$ depends only on a bicoset $\kappa\in\Sigma_{\ca{C}}\backslash \Sigma_n/(\Sigma_{k_1}\times\cdots\times\Sigma_{k_s})$, not on the permutation $\sigma$. Indeed, we are allowed to permute entries within each cluster, since this does not affect $\sigma(V_j)$ in the definition. On the other hand, we are allowed to permute entries in boxes, since this does not affect the numbers $k_i(V_j)$ in the definition. It can be seen that $X_{(\ca{C},\kappa)}$ is a face submanifold with the isotropy group
\[
T_{\sigma(\ca{C})}=\Delta_{\sigma(V_1)}(T^1)\times\cdots\times \Delta_{\sigma(V_s)}(T^1),
\]
for the diagonal maps $\Delta_{\sigma(V_j)}\colon T^1\to T^{\sigma(V_j)}$.

We define the map of posets $h\colon \ca{P}\to \ca{F}$, by setting $h((\ca{C},\kappa))=X_{(\ca{C},\kappa)}$. It is easily seen, that the map preserves the order. Minimal elements of $\ca{P}$ map bijectively to minimal elements of $\ca{F}$ according to item I. For each minimal element $x\in\ca{P}$ the poset $\ca{P}_{\geqslant x}$ maps isomorphically to $\ca{F}_{\geqslant h(x)}$, since both posets are isomorphic to $\Flats(\Gamma(k_1,\ldots,k_s))$. Since every face submanifold contains a fixed point, the latter fact implies the surjectivity of $h$. The injectivity of $h$ follows from the precise construction of $h$ and left as an exercise. This finishes the proof of item III. Theorem~\ref{thmMegaDescription} is now proved.
\end{proof}

\begin{rem}\label{remLookOfFacesGeneral}
It can be seen from the proof above, that the face submanifold $X_{(\ca{C},\kappa)}$ of $\Mbd'_{\Gamma,\lambda}$ is diffeomorphic to the direct product
\[
\prod_{j=1}^{l}\left(\prescript{\overline{k(V_j)}}{}{M'_{\Gamma,\lambda}}\right)
\]
of the manifolds having similar nature. Notice that some numbers $k_i(V_j)$ may vanish: the corresponding entries should be omitted. In particular, this shows that, in a partial flag variety some faces may be varieties of flags of smaller length.
\end{rem}

\begin{rem}\label{remDegenerateSpectra}
Notice that there is no twin analogue of $\Mbd'_{\Gamma,\lambda}$ defined via isospectral matrices. Twins are, by construction, the subsets of a full flag variety, while $\Mbd'_{\Gamma,\lambda}$ is generally a subset of the partial flag variety.

It is tempting to define $\Mbd_{\Gamma,\lambda}$ as the space of all isospectral matrices of certain shape, where eigenvalues have multiplicities $k_1,\ldots,k_s$, as was done in the proof of Proposition~\ref{propMatrMultiplicities}. However, if the matrices have zeroes at some positions, such space may not be a manifold, even for a generic nonsimple spectrum. This fact was observed by Tomei~\cite{Tomei:1984}, who noticed that the space of isospectral tridiagonal matrices with degenerate spectrum is never a manifold.
\end{rem}

\begin{rem}
Notice that the 1-skeleta of the posets~\eqref{eqPartFlagsPoset} and~\eqref{eqPartGraphPoset} are examples of Schreier coset graphs~\cite{Schreier:1927}. These are the analogues of Cayley graphs, in which the rights cosets $G/H$ are the vertices, and the edges correspond to multiplication by a fixed collection of generators of $G$. In our examples related to partial flags we use $G=\Sigma_n$ and $H=\Sigma_{k_1}\times\cdots\times\Sigma_{k_s}$. Therefore, the posets~\eqref{eqPartFlagsPoset} and~\eqref{eqPartGraphPoset} can be considered as higher-dimensional structures extending Schreier graphs.

More generally, the symmetric identification sets (Construction~\ref{conIdentSetBasic}) can be considered as higher dimensional generalizations of Schreier/Cayley graphs.
\end{rem}

The theory described in this section suggests the following generalization of Problem~\ref{probMainFormality}.

\begin{probl}\label{probPartial}
For which graphs $\Gamma$ and positive numbers $k_1,\ldots,k_s$, the torus action on a generic manifold $\Mbd'_{\Gamma,\lambda}$ is cohomologically equivariantly formal?
\end{probl}

Our general conjecture is that the answer to Problems~\ref{probMainFormality},\ref{probLieTypesGen},\ref{probPartial}, stated in this paper is the following. The manifolds of all the listed classes are equivariantly formal if and only if they are either projective varieties (admitting Bialynicki-Birula decomposition) --- in case of Problem~\ref{probPartial}, or the twins of such varieties --- in case of Problems~\ref{probMainFormality} and~\ref{probLieTypesGen}.

\section*{Acknowledgements}

The authors thank the anonymous referees for their comments on the first version of the paper.

\end{document}